\documentclass[10pt]{article}

\usepackage{geometry}
\usepackage{srcltx,graphicx}
\usepackage{amsmath,amssymb,amsthm,amscd,mathrsfs,bm,stmaryrd}
\usepackage[unicode,colorlinks,linkcolor=blue,hyperindex]{hyperref}
\usepackage{indentfirst} 
\usepackage{enumerate} 
\usepackage{cases} 
\usepackage[mathscr]{euscript}
\usepackage{url,breakurl}
\usepackage{wasysym}
\usepackage{algorithm}
\usepackage{algpseudocode}
\usepackage{tikz-cd}
\usepackage{arydshln}
\usepackage{colortbl}

\usepackage{yhmath}
\usepackage{extarrows}
\usepackage{booktabs}

\usepackage[T3,T1]{fontenc}
\DeclareSymbolFont{tipa}{T3}{cmr}{m}{n}
\DeclareMathAccent{\invbreve}{\mathalpha}{tipa}{16}

\usepackage{subcaption}
\makeatletter
\renewcommand{\p@subfigure}{\thefigure}
\makeatother

\numberwithin{equation}{section} \numberwithin{figure}{section} \numberwithin{table}{section}

{\theoremstyle{remark} \newtheorem{remark}{Remark}}
\newtheorem{theorem}{Theorem}[section]
\newtheorem{corollary}[theorem]{Corollary}
\newtheorem{lemma}[theorem]{Lemma}

\def\XXint#1#2#3{{\setbox0=\hbox{$#1{#2#3}{\int}$ }
\vcenter{\hbox{$#2#3$ }}\kern-.6\wd0}}

\makeatletter
\def\widebreve{\mathpalette\wide@breve}
\def\wide@breve#1#2{\sbox\z@{$#1#2$}%
     \mathop{\vbox{\m@th\ialign{##\crcr
\kern0.08em\brevefill#1{0.8\wd\z@}\crcr\noalign{\nointerlineskip}%
                    $\hss#1#2\hss$\crcr}}}\limits}
\def\brevefill#1#2{$\m@th\sbox\tw@{$#1($}%
  \hss\resizebox{#2}{\wd\tw@}{\rotatebox[origin=c]{90}{\upshape(}}\hss$}
\makeatletter

\newcommand{\triplenorm}[1]{\ensuremath{|\!|\!| #1 |\!|\!|}}

\newcommand{\curl}{{\rm curl}}
\newcommand{\rot}{{\rm rot}}
\renewcommand{\div}{{\rm div}}
\newcommand{\divg}{{\rm div}_\gamma}
\newcommand*{\dif}{\mathop{}\!\mathrm{d}}

\newcommand{\tr}{{\rm tr}}

\newcommand{\bz}{{\boldsymbol z}}
\newcommand{\bn}{{\boldsymbol n}}
\newcommand{\bt}{{\boldsymbol t}}
\newcommand{\bv}{{\boldsymbol v}}

\newcommand{\bP}{{\mathbf{P}}}
\newcommand{\bA}{{\mathbf{A}}}
\newcommand{\bH}{{\mathbf{H}}}
\newcommand{\bnu}{{\bm{\nu}}}
\newcommand{\Hg}{{\mathrm{H}_\gamma}}
\newcommand{\HG}{{\mathrm{H}_{\Gamma_h}}}
\newcommand{\nbg}{{\nabla_\gamma}}
\newcommand{\nbG}{{\nabla_{\Gamma_h}}}
\newcommand{\bp}{{\boldsymbol p}}
\newcommand{\bq}{{\boldsymbol q}}

\begin{document} 
\title{Stabilized Morley FEM for Surface Stokes in Stream-Function Formulation: Optimal Convergence via a New Geometric Estimate}
\author{
Shuonan Wu\thanks{Corresponding author. School of Mathematical Sciences,
Peking University, Beijing 100871, China
(\href{mailto:snwu@math.pku.edu.cn}{snwu@math.pku.edu.cn}).}
\and
Hao Zhou\thanks{School of Mathematical Sciences, Peking University,
Beijing 100871, China
(\href{mailto:zhouhao23@pku.edu.cn}{zhouhao23@pku.edu.cn}).}
}
\date{}
\maketitle

\begin{abstract}
We propose an intrinsic stabilized Morley finite element method for the stream-function formulation of the surface Stokes problem on closed surfaces. The method is posed directly on a polyhedral approximation of the surface. A parameter-free jump stabilization is introduced to recover coercivity, and a discrete Korn's inequality for the trace-free Hessian is established.

The main analytical contribution is a normal-separated geometric estimate for piecewise linearly approximated closed surfaces. It shows that a broad class of normal-dependent geometric consistency errors is in fact second order, even though a direct treatment suggests only first-order control. The missing order is recovered through an integral cancellation in the first normal variation. This estimate refines the standard $\bP_h\bnu$-type estimate and yields second-order consistency for the trace-free Hessian, the Stokes-type tensor Green identity, and the conormal fluxes arising in the discretization. Together with the discrete Korn's inequality, these estimates yield, under the natural $H^3$ regularity, optimal-order convergence: first order in the broken $H^2$ norm and second order in the broken $H^1$ norm. As a direct consequence, the recovered tangential velocity admits a second-order $L^2$ estimate. Numerical experiments are provided to support the theoretical results.
\end{abstract}




\section{Introduction}\label{sec:intro}
The surface Stokes equations describe stationary incompressible viscous flow constrained to a smooth surface \cite{jankuhn2018incompressible}. Their finite element approximation combines the velocity--pressure stability issues of the planar Stokes problem with geometric constraints: the discrete velocity must represent tangential flow on an approximated surface, and the differential operators themselves depend on the surface geometry. Surface finite element methods (SFEMs) and trace finite element methods (TraceFEMs) provide standard fitted and unfitted frameworks for these problems \cite{dziuk2013finite}.

Most numerical methods retain velocity and pressure as the primary unknowns. TraceFEM and parametric SFEM schemes based on ambient $H^1$-conforming vector spaces impose tangentiality weakly \cite{olshanskii2018finite,jankuhn2021error,brandner2022finite,hardering2024parametric}. Other constructions enforce tangentiality strongly: surface $H(\div)$-conforming and discontinuous spaces are used in \cite{bonito2020divergence,lederer2020divergence}, while penalty-free nodal MINI and Taylor--Hood spaces are developed in \cite{demlow2024tangential,demlow2025taylor} by Demlow and Neilan. Nochetto and Shakipov retain the velocity--pressure variables but replace the saddle-point formulation by an elliptic system that avoids the usual discrete inf--sup condition \cite{nochetto2025surface}. These approaches discretize the original vector system and provide velocity and pressure directly.

On a simply connected closed surface, a different formulation is available. Every tangential divergence-free velocity has a unique zero-mean stream function $\phi$ such that $\bm u=\curl_\gamma\phi$; testing the momentum equation with surface curl then eliminates the pressure \cite{reusken2020stream}. The resulting scalar problem is
\begin{equation} \label{eq:s-streamStokes}
\frac{1}{2}\Delta_\gamma^2 \phi +\div_\gamma((K-1)\nabla_\gamma \phi)  = -\rot_\gamma \bm{f} \quad \text{on }\gamma, \quad \int_\gamma \phi \,\mathrm{d}\sigma = 0,
\end{equation}
where $\Delta_\gamma$ is the Laplace--Beltrami operator and $K$ is the Gaussian curvature. The simple-connectivity assumption is essential for this representation: on surfaces of nontrivial topology, harmonic velocity components must also be included \cite{reusken2020stream,brueers2026releasing}. 

The stream-function formulation removes both the pressure and the divergence constraint, but replaces the vector saddle-point problem by a scalar fourth-order equation. Existing discretizations of this formulation either split it into two coupled second-order scalar equations \cite{reusken2020stream,brandner2020finite} or treat the fourth-order problem directly with a $C^0$ interior penalty method \cite{neilan2024c}. A systematic numerical comparison with velocity--pressure discretizations is given in \cite{brandner2022finite}. More broadly, standard primal Galerkin discretizations of fourth-order surface PDEs would require $H^2$-conforming (that is, $C^1$) finite element spaces, which cannot be constructed on a Lipschitz polyhedral surface. This has motivated $C^0$-type primal schemes \cite{larsson2017continuous,cai2024continuous}, second-order splitting and mixed formulations \cite{elliott2019second,stein2019mixed}, and the surface Hellan--Herrmann--Johnson method \cite{walker2022kirchhoff}.

This paper develops a stabilized primal nonconforming discretization of \eqref{eq:s-streamStokes} based on the surface Morley element \cite{morley1968triangular,wang2006morley}. Related stabilized nonconforming methods for the surface biharmonic equation have been developed using Zienkiewicz-type elements \cite{wu2025stabilized}. A distinguishing advantage of the Morley construction is its intrinsic character: for a fixed triangulation, the method is determined entirely by the edge lengths of its triangles and does not require knowledge of how the discrete surface is embedded in three-dimensional space. A related perspective is discussed in \cite{Yakov2025blowup}. The discrete bilinear form uses no information about the Gaussian curvature of the exact surface $\gamma$, nor any numerical approximation of it, and includes a fixed value-jump stabilization with no tunable penalty parameter.

The main analytical challenge concerns geometric consistency. Although a piecewise linear surface approximates positions and areas with second-order accuracy, its normal is only first-order accurate pointwise \cite{dziuk2013finite,demlow2009higher}. Consequently, a direct treatment of normal-dependent terms gives only first-order consistency. For surface Stokes in the stream-function formulation, expressed through a $\bP_h\bnu$-type estimate, this consistency error improves a geometric error bound of order $\mathcal{O}(h^{k_g})$ to
$\mathcal{O}\bigl(h^{\min\{k_g+1,\,2k_g-1\}}\bigr)$ (cf.~\cite{neilan2024c}),
where $k_g$ is the degree of the geometric approximation.
At the lowest order, $k_g=1$, corresponding to a piecewise linear surface, this bound for the broken $H^1$ error nevertheless remains only first order. For the more tractable surface biharmonic problem, whose variational form involves the Laplacian rather than the trace-free Hessian, Larson and Larsson \cite{larsson2017continuous} recover second-order geometric consistency, but at the expense of additional $H^4$ regularity of the solution. Through a duality argument, however, this only delivers a second-order $L^2$ estimate and does not reach the broken $H^1$ norm.

The main analytical contribution is to overcome this remaining barrier through a new {\it normal-separated geometric consistency} argument. Surface differential operators can be expressed in terms of ambient derivatives and the surface normal, so their evaluation on the discrete surface can be interpreted as a perturbation of the normal variable. A Taylor expansion with respect to this variable reveals an integral cancellation in the first-order term, thereby recovering the missing order. This viewpoint provides a common framework that contains the well-known $\bP_h\bnu$ estimate as a special case. Applying the argument sharpens the estimates obtained in earlier analyses and yields the optimal convergence rates: first-order convergence in the broken $H^2$ norm and second-order convergence in the broken $H^1$ norm. As a direct consequence, the recovered velocity $\bm u_h=\curl_{\Gamma_h}\phi_h$ admits a second-order $L^2$ estimate, all under the minimal $H^3$ regularity assumption.

The remainder of the paper is organized as follows. Section~\ref{sec:preliminaries} introduces the surface Stokes problem, its stream-function formulation, and the geometric setting. Section~\ref{sec:geo-error} establishes the normal-separated geometric estimate and its vector, tensor, and conormal-flux consequences. Section~\ref{sec:FEM} defines the surface Morley space, the surface Piola interpolation, and the stabilized method. Section~\ref{sec:analysis} proves the broken-$H^2$ and broken-$H^1$ error estimates and derives the velocity $L^2$ estimate. Section~\ref{sec:numerical} presents the numerical experiments. Section~\ref{sec:conclusion} concludes with a discussion of the geometric estimate.

We shall use $X \lesssim Y$ (resp. $X \gtrsim Y$) to denote $X \leq CY$ (resp. $X \geq CY$), where $C$ is a constant independent of the mesh size $h$. We write $X \eqsim Y$ when both $X \lesssim Y$ and $X \gtrsim Y$ hold.

\section{Notation and preliminaries} \label{sec:preliminaries}

In this paper, we assume that $\gamma \subset \mathbb{R}^3$ is a compact, closed, orientable surface with regularity at least $C^2$. For sufficiently small $\delta>0$, there exists a tubular neighborhood
$
U_\delta := \{ x \in \mathbb{R}^3 : \mathrm{dist}(x,\gamma) < \delta \},
$
in which the signed distance function $d$ is well-defined, with $d>0$ outside $\gamma$. The unit outward normal is given by $\bm{\nu} := \nabla d$, and the Weingarten map by $\mathbf{H} := \nabla^2 d$. The closest point projection $\bm{p}: U_\delta \to \gamma$ is defined by
$
\bm{p}(x) := x - d(x)\nabla d(x).
$
An important property is that $\bnu(x) = \bnu(\bm{p}(x))$ for $x \in U_\delta$ \cite[eq. (24)]{bonito2020finite}.
The tangential projection is $\mathbf{P} := \mathbf{I} - \bm{\nu} \otimes \bm{\nu}$, where $\mathbf{I}$ is the $3\times 3$ identity matrix. Consequently,
$
\nabla \bm{p} = \mathbf{P} - d\,\mathbf{H}.
$

\subsection{Differential operators and function spaces}
For a scalar function \(\psi\) on \(\gamma\), let $\psi^e$ denote a smooth extension of $\psi$ to $U_\delta$ (for instance $\psi^e:=\psi\circ\bm{p}$); the surface operators defined below are independent of the particular extension, since they only involve tangential derivatives. The tangential gradient is \(\nabla_\gamma \psi := \mathbf{P}\nabla \psi^e\),  
and the surface curl is \(\curl_\gamma \psi := \bm{\nu}\times\nabla_\gamma \psi\).
For a vector field \(\bm{g}=(g_1,g_2,g_3)^\top\), define \(\nabla\bm{g}^e:=(\nabla g_1^e,\nabla g_2^e,\nabla g_3^e)^\top\). Its surface Jacobian is \(\nabla_\gamma \bm{g}:=\mathbf{P}\nabla\bm{g}^e\mathbf{P}\). The surface deformation tensor is defined as
\[
\mathrm{E}_\gamma \bm{g} = \frac{1}{2}(\nabla_\gamma \bm{g}+(\nbg \bm{g})^\top).
\]

The surface divergence of a vector field is defined by \(\div_\gamma \bm{g}:=\mathrm{tr}(\nabla_\gamma \bm{g})\), and for a matrix field  $\bm{\tau}=(\bm{\tau}_1,\bm{\tau}_2,\bm{\tau}_3)^\top$, \(\div_\gamma \bm{\tau}:=(\div_\gamma\bm{\tau}_1,\div_\gamma\bm{\tau}_2,\div_\gamma\bm{\tau}_3 )^\top \). The scalar curl is \(\rot_\gamma \bm{g}:=\div_\gamma(\bm{g}\times\bm{\nu})\).  
The Laplace--Beltrami operator is \(\Delta_\gamma \psi:=\div_\gamma\nabla_\gamma \psi\), and the surface Hessian is \(\nabla_\gamma^2 \psi:=\nabla_\gamma(\nabla_\gamma \psi)\).
Define the trace-free Hessian operator as
\[
\Hg \psi := \mathrm{E}_\gamma(\curl_\gamma \psi).
\]
Its relation to the surface Hessian (cf.~\cite[(2.7)]{neilan2024c}) is
\begin{equation}\label{eq:H-Hessian}
\Hg \psi = \tfrac{1}{2}\bigl(\bnu^\times \nabla_\gamma^2 \psi - \nabla_\gamma^2 \psi\, \bnu^\times\bigr)=\text{sym}(\bnu^\times \nabla_\gamma^2 \psi),
\end{equation}
where $\bnu^\times$ denotes the skew-symmetric tensor defined by its action $\bnu^\times \cdot := \bnu \times \cdot$ for any vector or matrix. 


\smallskip
The standard notation $W_q^m(\gamma)$ is used for the Sobolev space of order $m$ and exponent $q$ on $\gamma$, equipped with the seminorm $|\cdot|_{W_q^m(\gamma)}$ and the norm $\|\cdot\|_{W_q^m(\gamma)}$. When $m=0$, this space is denoted by $L^q(\gamma)$. In particular, we define the Hilbert space $H^m(\gamma):=W_2^m(\gamma)$ and denote the $L^2$ inner product on $\gamma$ by $(\cdot,\cdot)_\gamma$. Similar notation applies to any subdomain of $\gamma$.
%
%
Let $\mathring{W}_q^m(\gamma)$ denote the subspace of $W_q^m(\gamma)$ consisting of functions with zero mean. In particular, $\mathring{L}^2(\gamma)$ denotes the subspace of $L^2(\gamma)$ consisting of functions with vanishing mean value. 
For any domain $D$ considered below, we write $\bm{W}_q^m(D):=[W_q^m(D)]^3$,
with the componentwise norm, and set $\bm{H}^m(D):=\bm{W}_2^m(D)$ and
$\bm{L}^q(D):=\bm{W}_q^0(D)$.
%
We also define the space
$$
\bm{H}({\rm div}_\gamma; \gamma) = \{\bm{g} \in \bm{L}^2(\gamma):~ {\rm div}_\gamma \bm{g} \in L^2(\gamma)\}.
$$


\subsection{Surface Stokes problem}
The surface Stokes problem seeks the fluid velocity $\bm{u} : \gamma \to \mathbb{R}^3$ satisfying $\bm{u} \cdot \bm{\nu} = 0$ and the surface fluid pressure $p : \gamma \to \mathbb{R}$ such that
\begin{subequations} \label{eq:strong_stokes}
\begin{align}
-\mathbf{P} \div_\gamma(\mathrm{E}_\gamma(\bm{u})) + \bm{u} + \nabla_\gamma p &= \bm{f} \quad \text{on } \gamma, \label{eq:strong_stokes_a}\\
\div_\gamma \bm{u} &= 0 \quad \text{on } \gamma, \label{eq:strong_stokes_b}
\end{align}
\end{subequations}
where $\bm{f} \in \bm{L}^2(\gamma)$ with $\bm{f} \cdot \bm{\nu} = 0$ is a given applied force vector.

Assuming the surface $\gamma$ is \textit{simply connected}, there exists a unique stream function $\phi \in \mathring{H}^2(\gamma)$ such that $\bm{u} = \curl_\gamma \phi$. To derive the variational formulation, we take the inner product of \eqref{eq:strong_stokes_a} with a test function $\bm{v} = \curl_\gamma \psi$ for $\psi \in \mathring{H}^2(\gamma)$. The pressure term vanishes due to $\div_\gamma(\curl_\gamma \psi) = 0$. Applying integration by parts to the viscous term yields
\begin{equation*}
(-\mathbf{P} \div_\gamma(\mathrm{E}_\gamma(\bm{u})), \curl_\gamma \psi)_\gamma = (\mathrm{E}_\gamma(\bm{u}), \nabla_\gamma \curl_\gamma \psi)_\gamma = (\Hg \phi, \Hg \psi)_\gamma.
\end{equation*}
The lower-order term simplifies to $(\bm{u}, \curl_\gamma \psi)_\gamma = (\nabla_\gamma \phi, \nabla_\gamma \psi)_\gamma$.

Consequently, the weak formulation of the surface Stokes problem seeks $\phi \in \mathring{H}^2(\gamma)$ such that
\begin{equation} \label{eq:weak_form}
a(\phi, \psi) = l(\psi) \quad \forall \psi \in \mathring{H}^2(\gamma),
\end{equation}
where the bilinear form $a(\cdot, \cdot)$ and the linear functional $l(\cdot)$ are defined as
\begin{align}
a(\phi, \psi) &:= (\Hg \phi, \Hg \psi)_\gamma + (\nabla_\gamma \phi, \nabla_\gamma \psi)_\gamma, \label{eq:bilinear_form}\\
l(\psi) &:= (\bm{f}, \curl_\gamma \psi)_\gamma. \label{eq:linear_functional}
\end{align}
The well-posedness of \eqref{eq:weak_form} follows from the Lax--Milgram lemma. In particular, the coercivity on $\mathring{H}^2(\gamma)$ is guaranteed by the surface Korn's inequality for $\curl_\gamma \phi$ and the Poincaré inequality. 

Furthermore, we can explicitly connect the weak form \eqref{eq:weak_form} to a scalar fourth-order PDE. Recall the surface differential identity \cite[(2.13)]{reusken2020stream}:
\begin{equation*}
\mathbf{P} \div_\gamma(\mathrm{E}_\gamma(\bm{u})) =  \nabla_\gamma \div_\gamma \bm{u} + K \bm{u}+\frac{1}{2} \curl_\gamma (\rot_\gamma \bm{u}),
\end{equation*}
where $K$ is the Gaussian curvature of $\gamma$. For sufficiently smooth functions $\phi$ and $\psi$, applying this identity and performing integration by parts allows us to rewrite the bilinear form and the linear functional as
\begin{equation} \label{eq:l_continue}
\begin{aligned}
a(\phi,\psi) &= - (\div_\gamma (\Hg \phi), \curl_\gamma \psi)_\gamma + (\nabla_\gamma \phi, \nabla_\gamma \psi)_\gamma \\
&= ( \frac{1}{2}\Delta_\gamma^2 \phi +\div_\gamma(K\nabla_\gamma \phi) - \Delta_\gamma \phi, \psi  )_\gamma, \\
l(\psi) &= (-\rot_\gamma \bm{f}, \psi)_\gamma.
\end{aligned}
\end{equation}
Equating these two terms yields the strong form of the surface stream-function equation \eqref{eq:s-streamStokes}.


Regarding the regularity of the solution, on a $C^4$ surface $\gamma$, employing a partition of unity together with interior elliptic regularity estimates yields the bound
\begin{equation} \label{eq:regularity}
\|\phi\|_{H^{3}(\gamma)} \lesssim \|\rot_\gamma \bm{f}\|_{H^{-1}(\gamma)} \lesssim \|\bm{f}\|_{L^2(\gamma)} .
\end{equation}
This regularity result can also be found in \cite[Lemma 6.2]{neilan2024c}.

\subsection{Discretization}
Let $\Gamma_h$ be a polyhedral surface approximation of $\gamma$ composed of shape-regular, quasi-uniform triangular faces $\mathcal{T}_h$ with mesh size $h := \max_{K \in \mathcal{T}_h} \mathrm{diam}(K)$. Assume $h$ is sufficiently small such that $\Gamma_h \subset U_\delta$ and the closest point projection $\bm{p}: \Gamma_h \to \gamma$ is bijective and yields an $\mathcal{O}(h^2)$ geometric approximation. The projected exact surface elements and the corresponding mesh are denoted by
$
K^\gamma := \bm{p}(K)$, and $\mathcal{T}_h^\gamma := \{K^\gamma: K \in \mathcal{T}_h\}.
$

Let $\mathcal{V}_h$ and $\mathcal{E}_h$ denote the sets of all vertices and edges in $\mathcal{T}_h$. For a given $K \in \mathcal{T}_h$, its vertices and edges are denoted by $\mathcal{V}_K$ and $\mathcal{E}_K$. For an interior edge $e = \partial K_1 \cap \partial K_2 \in \mathcal{E}_h$, the discrete and exact edge patches are defined as $\omega_e := K_1\cup K_2$ and $\omega_{e^\gamma} := K_1^\gamma \cup K_2^\gamma$, respectively.

The piecewise constant outward unit normal to $\Gamma_h$ is denoted by $\bm{\nu}_h$, with its restriction to $K$ being $\bm{\nu}_K := \bm{\nu}_h|_K$. The discrete tangential projection is $\mathbf{P}_h := \mathbf{I} - \bm{\nu}_h \otimes \bm{\nu}_h$. The surface measures of $\gamma$ and $\Gamma_h$ are related by $\mathrm{d}\sigma(\bm{p}(x)) = \mu_h(x) \mathrm{d}\sigma_h(x)$, where $|1 - \mu_h| \lesssim h^2$. For simplicity, we frequently omit the composition with $\bm{p}$ and write $|\bm{\nu} - \bm{\nu}_h| \lesssim h$.

On the exact surface element $K^\gamma$, let $\bm{n}$ denote the outward unit co-normal vector along $\partial K^\gamma$, and let $\bm{t} := \bm{\nu} \times \bm{n}$ be the unit tangent vector along $\partial K^\gamma$. Analogously, for a discrete element $K \in \mathcal{T}_h$, let $\bm{n}_K$ denote the co-normal along $\partial K$, and $\bm{t}_K := \bm{\nu}_K \times \bm{n}_K$ be the edge tangent vector. When defined globally on $\Gamma_h$, these piecewise constant vectors are denoted by $\bm{n}_h$ and $\bm{t}_h$.

\paragraph{Operators and function spaces}
Differential operators on $\Gamma_h$, such as $\nabla_{\Gamma_h}$, ${\rm div}_{\Gamma_h}$, $\curl_{\Gamma_h}$ and $\HG$, are understood in a piecewise sense. Accordingly, we define the broken Sobolev spaces:
$$
H_h^m(\Gamma_h) := \{v \in L^2(\Gamma_h): v|_K \in H^m(K),\ \forall K \in \mathcal{T}_h\}, \quad \|v\|_{H^m_h(\Gamma_h)}^2 := \sum_{K \in \mathcal{T}_h} \|v\|_{H^m(K)}^2.
$$
For a domain $D \in \{\Gamma_h, \mathcal{E}_h, K, e\}$, denote the $L^2$ inner product on $D$ by $ (\cdot, \cdot)_D $.
We also write
\[
H_h^m(\gamma)
=
\{v\in L^2(\gamma): v|_{K^\gamma}\in H^m(K^\gamma)\ \forall K\in\mathcal T_h\},\quad \|v\|_{H_h^m(\gamma)}^2
=
\sum_{K\in\mathcal T_h}\|v\|_{H^m(K^\gamma)}^2.
\]
The vector-valued broken space $\bm{H}_h^m(\Gamma_h):=[H_h^m(\Gamma_h)]^3$ is
defined componentwise, with the corresponding componentwise norm.

Define the second-order directional derivatives of $v\in H_h^2(K)$ along $\partial K$ as
\begin{equation*} 
\partial_{tt} v := \bm{t}_K^\top (\nabla_K^2 v) \bm{t}_K, \quad
\partial_{nn} v := \bm{n}_K^\top (\nabla_K^2 v) \bm{n}_K, \quad
\partial_{tn} v := \bm{t}_K^\top (\nabla_K^2 v) \bm{n}_K.
\end{equation*}

Next, we define the standard jump and average operators. For edge $e=\partial K_1 \cap \partial K_2$, let $\bm{n}_i$ denote the co-normal vector to $\partial K_i$ along $e$, and let $\bm{t}_i := \bm{\nu}_i \times \bm{n}_i$ be the corresponding unit tangent vector. Because $\Gamma_h$ is a discrete polyhedral surface, $\bm{n}_1 \neq - \bm{n}_2$ in general, whereas $\bm{t}_1 = - \bm{t}_2$. 
For a scalar trace $v_i:=v|_{K_i}$, and an ordinary vector-valued trace $\bm v_i:=\bm v|_{K_i}$, we use the value jump and average
\begin{equation}\label{eq:ordinary-jump-average}
        [v] := v_1-v_2,
        \quad
        \{v\}:=\frac12(v_1+v_2),\quad \llbracket \bm v \rrbracket := \bm v_1-\bm v_2,
        \quad
        \{\bm v\}:=\frac12(\bm v_1+\bm v_2).
\end{equation}
For normal and tangential components, we use the outward co-normal convention
\begin{equation}\label{eq:component-jump-average}
\begin{aligned}
[\bm v\cdot\bm n_h]
:= \bm v_1\cdot\bm n_1+\bm v_2\cdot\bm n_2,&\qquad
[\bm v\cdot\bm t_h]
:= \bm v_1\cdot\bm t_1+\bm v_2\cdot\bm t_2, \\
\{\bm v\cdot\bm n_h\}
:=\frac12(\bm v_1\cdot\bm n_1-\bm v_2\cdot\bm n_2),&\qquad 
\{\bm v\cdot\bm t_h\}
:=\frac12(\bm v_1\cdot\bm t_1-\bm v_2\cdot\bm t_2).
\end{aligned}
\end{equation}
Here the minus sign in the averages, and the plus sign in the jumps, reflect the outward co-normal convention $\bm n_1,\bm n_2$ (and $\bm t_1,\bm t_2$), which point in opposite directions across $e$ on the polyhedral surface; with this convention $[\bm v\cdot\bm n_h]$ measures the true normal-flux mismatch and $\{\bm v\cdot\bm n_h\}$ its average.
For $\bm{v} \in \bm{H}_h^1(\Gamma_h)$, $\bm{v} \in \bm{H}(\mathrm{div}_{\Gamma_h}; \Gamma_h)$ if and only if $[\bm{v} \cdot \bm{n}_h]|_e = 0$ for all interior edges $e$.
For an elementwise tensor field $\bA$, we define the
conormal-flux jump and average by
\begin{equation}\label{eq:tensor-conormal-jump-average}
\begin{aligned}
\llbracket \bA \bm n_h\rrbracket
:= \bA_1\bm n_1+\bA_2\bm n_2,\qquad 
\{\bA\bm n_h\}
:=\frac12(\bA_1\bm n_1-\bA_2\bm n_2),
\end{aligned}
\end{equation}
where $\bA_i:=\bA|_{K_i}$. With these conventions, for any
ordinary vector trace $\bm v$,
\begin{equation}\label{eq:edge-flux-decomposition}
\begin{aligned}
(\bA_1\bm n_1)\cdot\bm v_1
+
(\bA_2\bm n_2)\cdot\bm v_2
=
\{\bA_h\bm n_h\}\cdot\llbracket \bm v\rrbracket
+
\llbracket \bA_h\bm n_h\rrbracket\cdot\{\bm v\}.
\end{aligned}
\end{equation}

\begin{lemma}\label{lm:Pnjump}
Let $\llbracket \bn_h \rrbracket := \bm{n}_1 + \bm{n}_2$ denote the jump of the discrete co-normal across an interior edge $e$. Then, the exact tangential projection satisfies
$$
|\bP \llbracket \bn_h \rrbracket |\lesssim h^2.
$$
\end{lemma}
The proof is deferred to Appendix \ref{sec:appendix_proof}.

\paragraph{Extensions and lifts} 
The closest point projection $\bm{p}: \Gamma_h \to \gamma$ is a bijection, with its inverse denoted by $\bm{p}^{-1}$. For a function $v$ on $\Gamma_h$, its lift to $\gamma$ is defined as $v^\ell := v \circ \bm{p}^{-1}$. Conversely, for a function $v \in H^m(\gamma)$ with its extension $v^e = v \circ \bm{p}$, it holds that
\begin{equation}\label{eq:nabla-trans}
\nabla v^e =(\nabla \bp)^T (\nabla_{\gamma} v)^e = (\bP-d\bH)(\nabla_{\gamma} v)^e,
\end{equation}
and, for each integer $0\le m\le 3$, the following norm equivalence holds provided that $\gamma$ is of class $C^{m+1}$ (cf. \cite{dziuk1988finite, bonito2020finite}):
\begin{equation}  \label{eq:scalar-norm-equiv}
\|v\|_{H^m(K^\gamma)} \simeq \|v^e\|_{H^m(K)} \quad \forall K \in \mathcal{T}_h.
\end{equation}

\section{A new geometric estimate with applications}\label{sec:geo-error}
In this section, we present a new geometric estimate for normal-dependent integrands. It provides a convenient way to recover the $\bP_h\bnu$-type estimate widely used in the surface finite element literature; see, for instance, \cite{larsson2017continuous,hansbo2020analysis,neilan2024c}. The new estimate applies to a broader class of geometric error terms.

The section is organized around a single master estimate.
We first prove Theorem~\ref{thm:normal-separated} (normal-separated
geometric estimate), a unified tool that converts a normal-dependent
integrand into a tangential defect plus a second-order remainder. Three
groups of applications follow. The applications to vector fields
(Section~\ref{subsec:geo-vector}) recover known results---the
$\bP_h\bnu$-type estimate and a first-order superconvergence
identity---and thus mainly illustrate the strength of the master
estimate. The remaining two groups produce the geometric consistency
estimates actually used in the error analysis of
Section~\ref{sec:analysis}; both concern the tangential pairing of a
stream function $\psi$ with a symmetric tangential tensor $\bA$, and are
distinguished by whether the estimate lives in the element interior or
on the skeleton. The interior consistency
(Section~\ref{subsec:geo-tensor}) culminates in the Stokes-type tensor
Green consistency (Corollary~\ref{cor:stokes-tensor-green}), assembled
from the trace-free Hessian consistency
(Lemma~\ref{lem:neilan-volume-estimate}) and the curl--divergence
consistency (Lemma~\ref{lem:stokes-curl-div}). The boundary-flux
consistency (Section~\ref{subsec:geo-boundary}) yields the conormal-flux
consistency (Corollary~\ref{cor:conormal-flux}), obtained by specializing
a weak conormal-flux estimate (Lemma~\ref{lem:weak-conormal-flux}).

\subsection{Normal-separated geometric estimate}
We begin with a fundamental observation regarding a higher-order geometric error estimate: although the normal vector of a polyhedral surface approximation generally satisfies only a first-order pointwise accuracy $\|\bnu_h-\bnu^e\|_{L^\infty(\Gamma_h)}= \mathcal{O}(h)$, its inner product with the extension of a tangential vector field gains an additional order of accuracy in the integral sense. Crucially, the principal constant of this enhanced bound is governed mainly by the surface divergence of the tangential vector field. The explicit form of this geometric estimate is formulated in the following lemma.

\begin{lemma}[normal-to-divergence estimate] \label{lm:normal-to-divergence}
Let $\Gamma_h$ be a polyhedral approximation to a closed surface $\gamma$. Assume $\bq$ is a tangential vector field on $\gamma$, i.e., $\bq \cdot \bnu = 0$ on $\gamma$. Then, the following geometric estimate holds:
\begin{equation}\label{eq:normal-to-divergence}
    \left| \int_{\Gamma_h} \bq^e \cdot \bnu_h \dif\sigma_h \right| \lesssim h^2 \|{\rm div}_\gamma \bq\|_{L^1(\gamma)} + h^3 \|\bq\|_{L^1(\gamma)}.
\end{equation}
\end{lemma}
\begin{proof}
First, a rough estimate shows that $|\bq^e \cdot \bnu_h| \lesssim h |\bq^e|$, since the tangential vector field $\bq$ satisfies $\bq^e \cdot \bnu_h= \bq^e \cdot (\bnu_h - \bnu^e)$. We provide a more refined estimate in the following steps. 

{\it Step 1 (decomposition of the normal component).} Since $\nabla d$ is constant along the closest point projection, it holds that $\nabla_{\Gamma_h} d = \bP_h \nabla d = \bP_h \bnu^e$. Taking the inner product with $\bq^e$, we have
\begin{equation*}
    \bq^e \cdot \nabla_{\Gamma_h} d = \bq^e \cdot (\bnu^e - (\bnu^e \cdot \bnu_h)\bnu_h) = - (\bnu^e \cdot \bnu_h)(\bq^e \cdot \bnu_h),
\end{equation*}
which gives a decomposition of $\bq^e \cdot \bnu_h$ into a principal part and a higher-order perturbation:
\begin{equation*}
    \bq^e \cdot \bnu_h = (\bq^e \cdot \bnu_h)(1 - \bnu^e \cdot \bnu_h) - \bq^e \cdot \nabla_{\Gamma_h} d.
\end{equation*}
For the first term, since $|1 - \bnu^e \cdot \bnu_h| \lesssim h^2$, combining this with the rough estimate $|\bq^e \cdot \bnu_h| \lesssim h |\bq^e|$ shows that its integral over $\Gamma_h$ is readily bounded by $C h^3 \|\bq\|_{L^1(\gamma)}$.

{\it Step 2 (representation of $\nabla_{\Gamma_h}d$).} Since $\nabla_{\Gamma_h}d$ depends solely on its values on $\Gamma_h$, we can regard $\bp$ as a mapping $\bp: \Gamma_h \to \gamma$. We define a function on $\gamma$ by $\rho_h(y) := d(\bp^{-1}(y))$ for $y\in \gamma$; equivalently, $d(x) = \rho_h(\bp(x))$ for all $x \in \Gamma_h$. Using the relation $\nabla \bp = \bP - d \bH$, the chain rule implies that on $\Gamma_h$,
\begin{equation} \label{eq:rho-h}
\bP_h \bnu^e = \nabla_{\Gamma_h} d = \bP_h (\bP - d \bH)(\nabla_\gamma\rho_h)^e.
\end{equation}
The tangential field satisfies $\bP (\nabla_\gamma\rho_h)^e = (\nabla_\gamma\rho_h)^e$. Since $|\bnu^e - \bnu_h| \lesssim h$, we obtain
\begin{equation*}
\bP_h (\bP - d \bH)(\nabla_\gamma\rho_h)^e = (\mathbf{I} - \bnu_h \otimes (\bnu_h - \bnu^e) - d\bP_h \bH) (\nabla_\gamma\rho_h)^e = (\mathbf{I} + \mathcal{O}(h)) (\nabla_\gamma\rho_h)^e.
\end{equation*}
Therefore, $|(\nabla_\gamma\rho_h)^e| \lesssim |\bP_h \bnu^e| \lesssim h$.

{\it Step 3 (estimate of $\int_{\Gamma_h} \bq^e \cdot \nabla_{\Gamma_h} d \dif\sigma_h$).} For a tangential vector field $\bq$, it holds that $\bP\bq^e = \bq^e$. Using the identity \eqref{eq:rho-h} from Step 2, we obtain
\begin{equation*}
\begin{aligned}
\bq^e \cdot \nabla_{\Gamma_h} d &= \bq^e \cdot \big(\bP_h (\bP - d \bH) (\nabla_\gamma\rho_h)^e\big) 
= \big( (\bP - d\bH) \bP_h \bq^e \big) \cdot (\nabla_\gamma \rho_h)^e \\
&= (\bP\bP_h\bP \bq^e - d\bH \bP_h \bq^e) \cdot (\nabla_\gamma \rho_h)^e \\
&= \bq^e \cdot (\nabla_\gamma \rho_h)^e + (\bP\bP_h\bP - \bP)\bq^e \cdot (\nabla_\gamma \rho_h)^e - d(\bH \bP_h \bq^e)\cdot (\nabla_\gamma \rho_h)^e.
\end{aligned}
\end{equation*}
Using the estimate $|(\nabla_\gamma \rho_h)^e| \lesssim h$ from Step 2, along with $|\bP\bP_h\bP - \bP| \lesssim h^2$ and $|d| \lesssim h^2$, the last two terms can be bounded pointwise by $C h^3 |\bq^e|$. Consequently, their integrals over $\Gamma_h$ are bounded by $C h^3 \|\bq\|_{L^1(\gamma)}$.

Applying integration by parts on the closed surface $\gamma$ to the first term yields
\begin{equation*}
\begin{aligned}
    \int_{\Gamma_h} \bq^e \cdot (\nabla_\gamma \rho_h)^e \dif\sigma_h &= \int_\gamma \bq \cdot \nabla_\gamma \rho_h \dif \sigma +  \int_{\gamma} \bq \cdot \nabla_\gamma \rho_h (\mu_h^{-1} - 1) \dif\sigma \\
    &= -\int_\gamma {\rm div}_\gamma \bq \cdot \rho_h \dif \sigma +  \int_{\gamma} \bq \cdot \nabla_\gamma \rho_h (\mu_h^{-1} - 1) \dif\sigma. 
\end{aligned}
\end{equation*}
Noting that $|\rho_h| \lesssim h^2$ and $|\mu_h^{-1} - 1| \lesssim h^2$, we obtain
\begin{equation*}
\left|\int_{\Gamma_h} \bq^e \cdot (\nabla_\gamma \rho_h)^e \dif\sigma_h\right| \lesssim h^2\|{\rm div}_\gamma \bq\|_{L^1(\gamma)} + h^3 \|\bq\|_{L^1(\gamma)}.
\end{equation*}
Combining the steps above yields the desired estimate \eqref{eq:normal-to-divergence}.
\end{proof}

\begin{theorem}[normal-separated geometric estimate]
\label{thm:normal-separated}
Let $\Gamma_h$ be a polyhedral approximation to a closed surface $\gamma$.
Let $F:\gamma\times\mathbb{S}^2 \to \mathbb{R}$, where the second variable can be extended to an open neighborhood $U$ of $\mathbb{S}^2$ (denoted as $\tilde{F}: \gamma \times U \to \mathbb{R}$).
For $y \in \gamma$, set
\begin{equation}\label{def:geo-qM}
\begin{aligned}
    F^\bnu(y) &:= F(y,\bnu(y)), \\ 
     \bq_F(y) &:=\bP(y)D_\bz \tilde{F}(y,\bnu(y)), \quad 
    q_F^\perp(y) := D_\bz \tilde{F}(y,\bnu(y)) \cdot \bnu(y), \\
    M_F(y) &:= \sup_{|\bz-\bnu(y)|\le C_\gamma h} |D_\bz^2 \tilde{F}(y,\bz)|,
\end{aligned}
\end{equation}
where $\bP(y) = \mathbf{I} - \bnu(y)\otimes \bnu(y)$ is the tangential projection, and the constant $C_\gamma$ is chosen large enough such that $|\bnu_h^\ell(y) - \bnu(y)| \le C_\gamma h$ for all $y \in \gamma$.
Assume $F^{\bnu}, q_F^\perp, M_F \in L^1(\gamma)$, $\bq_F \in \bm{L}^1(\gamma)$ and ${\rm div}_\gamma \bq_F \in L^1(\gamma)$. Then, the following geometric estimate holds:
\begin{equation}\label{eq:normal-separated}
\begin{aligned}
&\quad \left| \int_{\Gamma_h} F(y,\bnu_h^\ell(y))|_{y = \bp(x)} \dif \sigma_h(x) - \int_\gamma F(y,\bnu(y))\dif \sigma(y) \right| \\
&\lesssim h^{2}\big(
\|F^\bnu\|_{L^1(\gamma)} + \|{\rm div}_\gamma\bq_F\|_{L^1(\gamma)} + h\|\bq_F\|_{L^1(\gamma)} + \|q_F^\perp\|_{L^1(\gamma)} + \|M_F\|_{L^1(\gamma)}
\big).
\end{aligned}
\end{equation}
\end{theorem}

\begin{proof} 
Recall that $\bnu_h^\ell$ is the lift of $\bnu_h$ to $\gamma$. By expanding the function with respect to the second variable in the ambient space, we obtain
\begin{equation} \label{eq:geo-taylor}
    F(y,\bnu_h^\ell(y))-F(y,\bnu(y)) = D_\bz \tilde{F}(y,\bnu(y))\cdot(\bnu_h^\ell(y)-\bnu(y))+R_h(y) \qquad \text{for } y \in \gamma,
\end{equation}
where the second-order remainder $R_h$ satisfies $|R_h|\lesssim M_F|\bnu_h^\ell-\bnu|^2\lesssim h^2M_F$.

We decompose the gradient into its tangential and normal components: $D_\bz \tilde{F}(y,\bnu(y)) = \bq_F(y) + q_F^\perp(y)\bnu(y)$ (Here, we note that the tangential component $\bq_F$ is independent of the extension). Consequently, the linear term in \eqref{eq:geo-taylor} becomes
$$
\begin{aligned}
D_\bz \tilde{F}(y,\bnu) \cdot (\bnu_h^\ell-\bnu) &= \bq_F(y) \cdot (\bnu_h^\ell-\bnu) + q_F^\perp(y)\bnu \cdot (\bnu_h^\ell-\bnu) \\
&= \bq_F(y) \cdot (\bnu_h^\ell-\bnu) -\frac12 q_F^\perp(y)|\bnu_h^\ell - \bnu|^2.
\end{aligned}
$$
Therefore, we can absorb the normal component into a modified remainder $\tilde{R}_h(y) := R_h(y) - \frac{1}{2}q_F^\perp(y)|\bnu_h^\ell-\bnu|^2$, which is readily bounded by $|\tilde{R}_h| \lesssim h^2(M_F + |q_F^\perp|)$. The Taylor expansion \eqref{eq:geo-taylor} then simplifies to
$$
    F(y,\bnu_h^\ell(y))-F(y,\bnu(y)) = \bq_F(y)\cdot(\bnu_h^\ell(y)-\bnu(y))+\tilde{R}_h(y).
$$
Crucially, by our definition, $\bq_F$ is a tangential vector field on $\gamma$ (i.e., $\bq_F \cdot \bnu = 0$). We can now proceed with the integration:
\[
\begin{aligned}
&\int_{\Gamma_h}F(y,\bnu_h^\ell(y))|_{y = \bp(x)}\dif\sigma_h(x) 
- \int_\gamma F(y,\bnu(y))\dif\sigma(y) \\
=~&\int_{\gamma} \big(F(y,\bnu_h^\ell(y))- F(y,\bnu(y)) \big) \mu_h^{-1} \dif\sigma(y) + \int_\gamma(\mu_h^{-1}-1)F^\bnu\dif\sigma \\
=~&\int_{\gamma} \bq_F(y) \cdot \bnu_h^\ell(y) \mu_h^{-1}\dif\sigma(y) + \int_\gamma\mu_h^{-1} \tilde{R}_h\dif\sigma + \int_\gamma(\mu_h^{-1}-1)F^\bnu\dif\sigma \\
=~&\int_{\Gamma_h} \bq_F^e \cdot \bnu_h\dif\sigma_h + \int_\gamma\mu_h^{-1} \tilde{R}_h\dif\sigma + \int_\gamma(\mu_h^{-1}-1)F^\bnu\dif\sigma.
\end{aligned}
\]
Applying Lemma \ref{lm:normal-to-divergence} (normal-to-divergence estimate) directly to the first term and noting that $|\mu_h^{-1} - 1| \lesssim h^2$, we obtain the desired estimate \eqref{eq:normal-separated}.
\end{proof}

\begin{remark}[special case]
Let us take $\tilde{F}(y, \bz) = \bq(y) \cdot \bz$ in Theorem \ref{thm:normal-separated}, where $\bq$ is a tangential vector field on $\gamma$. Since this function is naturally linear with respect to $\bz$, a straightforward calculation yields $F^\bnu(y) = 0$, $\bq_F(y) = \bq(y)$, $q_F^\perp(y) = 0$, and $M_F(y) = 0$. Consequently, the estimate \eqref{eq:normal-separated} perfectly recovers the result established in Lemma \ref{lm:normal-to-divergence}.
\end{remark}

\begin{remark}[higher-order geometric approximation]
Suppose $\Gamma_h$ is a $k$-th order surface approximation satisfying the standard geometric bounds $|d| \eqsim |\rho_h| \lesssim h^{k+1}$, $|\bnu_h^\ell-\bnu| \lesssim h^k$, and $|\mu_h^{-1}-1| \lesssim h^{k+1}$ (cf. \cite{demlow2009higher}). Following the identical proof structure, the geometric estimate naturally generalizes to:
\begin{equation*}
\begin{aligned}
&\quad \left| \int_{\Gamma_h} F(y,\bnu_h^\ell(y))|_{y = \bp(x)}\dif \sigma_h(x) - \int_\gamma F(y,\bnu(y))\dif \sigma(y) \right| \\
& \lesssim
h^{k+1}\big(\|F^\bnu\|_{L^1(\gamma)} + \|{\rm div}_\gamma\bq_F\|_{L^1(\gamma)} + h^k\|\bq_F\|_{L^1(\gamma)}\big) + h^{2k}\big(\|q_F^\perp\|_{L^1(\gamma)} + \|M_F\|_{L^1(\gamma)}\big).
\end{aligned}
\end{equation*}
\end{remark}

\begin{remark}[intrinsic formulation in the fiber variable]
The preceding argument can also be formulated intrinsically on the trivial sphere bundle $\gamma\times\mathbb S^2$. Indeed, the second variable of $F$ is regarded as a point on the fiber $\mathbb S^2$, then the corresponding vertical gradient at $\bz=\bnu(y)$ lies in
\[
T_{\bnu(y)}\mathbb S^2
= \{\xi\in\mathbb R^3:\xi\cdot\bnu(y)=0\}
= T_y\gamma .
\]
Equivalently, after extending $F$ off the sphere by the radial projection $\bz\mapsto \bz/|\bz|$, namely $\tilde{F}(y, \bz) := F(y, \bz/|\bz|)$, the ambient derivative satisfies
\[
D_\bz \tilde{F}(y,\bnu(y))\cdot\bnu(y)=0 .
\]
Thus, in this case, the normal component $q_F^\perp$ vanishes and the estimate above holds with the term $\|q_F^\perp\|_{L^1(\gamma)}$ omitted.
\end{remark}

\subsection{Applications to vector fields}\label{subsec:geo-vector}

In this subsection, we demonstrate the applications of Theorem \ref{thm:normal-separated} to integral forms involving vector fields. As we shall see, these results recover, and in some cases strengthen, existing geometric error estimates.

\begin{corollary}[$\bP_h\bnu$-type estimate]
\label{cor:Phn}
For \(\bm{\chi}\in \bm{W}_1^1(\gamma)\) it holds that
\begin{equation}
\label{eq:Phn-estimate}
\left|
\int_{\Gamma_h}
\bP_h\bnu^e\cdot \bm{\chi}^e \dif\sigma_h
\right|
\lesssim
h^2\big(\|\div_{\gamma}(\bP\bm{\chi})\|_{L^1(\gamma)} + \|\bm{\chi}\|_{L^1(\gamma)}\big) \lesssim h^2 \|\bm{\chi}\|_{W_1^1(\gamma)}.
\end{equation}
\end{corollary}
\begin{proof}
Let us choose $\tilde{F}(y,\bz)=\bP_{\bz}\bnu(y)\cdot\bm{\chi}(y)$, where $\bP_{\bz}=\mathbf{I}-\bz\otimes\bz$. Noting that $\tilde{F}(y, \bz)$ is quadratic in $\bz$, a direct calculation according to the definitions in \eqref{def:geo-qM} yields
$$
F^\bnu(y) = 0, \quad \bq_F(y) = -\bP(y)\bm{\chi}(y), \quad q_F^\perp(y) = -2\bm{\chi}(y) \cdot \bnu(y), \quad M_F(y) \lesssim |\bm{\chi}(y)|.
$$
Moreover, $F(y,\bnu_h^\ell(y))|_{y = \bp(x)} = \bP_h\bnu^e(x)\cdot\bm{\chi}^e(x)$. Applying Theorem~\ref{thm:normal-separated} directly yields the first inequality in \eqref{eq:Phn-estimate}. The second inequality then follows from $\|\bP\|_{W^1_\infty(\gamma)} \lesssim 1$.
\end{proof}

It is worth noting that the first inequality in Corollary \ref{cor:Phn} provides a sharper geometric bound. A simple upper bound then yields the second inequality, which is the conventional estimate widely adopted in \cite{larsson2017continuous,hansbo2020analysis,neilan2024c}. Furthermore, we consider an estimate that plays a central role in optimal duality arguments for the Laplace-Beltrami surface FEM. While classical approaches often rely on pointwise superconvergence properties—such as $|\bP\bP_h\bP-\bP|\lesssim h^2$ (cf.~\cite{demlow2009higher})—to control tangential gradients, the normal-separated geometric estimate recovers this crucial result without imposing additional regularity. 

\begin{corollary}[superconvergence estimate for a first-order tangential form]
\label{cor:first-order-tangential}
Let $u\in H^1(\gamma)$, and $\bm{\chi} \in \bm{L}^2(\gamma)$ be a tangential vector field. Then
\begin{equation}
\label{eq:first-order-tangential}
\left|
(\nbG u^e,\bm{\chi}^e)_{\Gamma_h}
- (\nbg u,\bm{\chi})_\gamma \right|
\lesssim
h^2 \|\nbg u\|_{L^2(\gamma)} \|\bm{\chi}\|_{L^2(\gamma)} .
\end{equation}
\end{corollary}

\begin{proof}
Set $\tilde{F}(y,\bz)=\bP_{\bz}\nabla u^e(y)\cdot \bm{\chi}(y)$, where $\bP_{\bz}=\mathbf{I}-\bz\otimes\bz$.
On $\gamma$, we have $\nabla u^e = (\bP - d\bH)(\nabla_\gamma u)^e = \nbg u$, hence both $\nabla u^e$ and $\bm{\chi}$ are tangential vector fields. Therefore, a direct calculation from the definitions in \eqref{def:geo-qM} gives
$$
F^\bnu(y) = \nbg u(y) \cdot \bm{\chi}(y), \quad \bq_F(y) = \bm{0}, \quad q_F^\perp(y) = 0, \quad M_F(y) \lesssim |\nbg u(y)|\,|\bm{\chi}(y)|.
$$
Applying Theorem~\ref{thm:normal-separated} and the Cauchy-Schwarz inequality directly gives
\[
\left|
\int_{\Gamma_h}
\bP_h(\nbg u)^e\cdot\bm{\chi}^e\dif\sigma_h
- \int_\gamma \nbg u\cdot \bm{\chi}\dif\sigma \right|
\lesssim
h^2 \|\nbg u\|_{L^2(\gamma)} \|\bm{\chi}\|_{L^2(\gamma)} .
\]
Finally, by $\nabla_{\Gamma_h}u^e=\bP_h(\bP-d\bH)(\nabla_\gamma u)^e$, $|d|\lesssim h^2$, and the norm equivalence, we have 
\[
\left|
\int_{\Gamma_h}
\left(
\nbG u^e-\bP_h(\nbg u)^e
\right)\cdot\bm{\chi}^e\dif\sigma_h
\right|\lesssim
h^2
\|\nbg u\|_{L^2(\gamma)}
\|\bm{\chi}\|_{L^2(\gamma)}.
\]
Combining the two estimates proves \eqref{eq:first-order-tangential}.
\end{proof}

\subsection{Applications to tangential tensors: interior consistency}\label{subsec:geo-tensor}
Guided by the preceding analysis (e.g., Corollary \ref{cor:first-order-tangential}), the construction of $\tilde{F}$ follows a systematic rule: within the definitions of the extended surface differential operators, we replace all the $\bnu$ with $\bz$, thereby explicitly isolating the normal-dependent components.
In this subsection, we present several interior consistency estimates obtained from Theorem \ref{thm:normal-separated} for second-order tensors; together they yield the Stokes-type tensor Green consistency directly applicable to the geometric error analysis of the surface Stokes equations in stream-function formulation. The first application recovers \cite[Lemma 5.4, (5.7)]{neilan2024c} in the present notation.

\begin{lemma}[surface trace-free Hessian consistency] \label{lem:neilan-volume-estimate}
Let $\gamma$ be a $C^3$ surface, let $\bA$ be a tangential second-order tensor on \(\gamma\), with components in \(H^1(\gamma)\), and let \(\psi\in H^3(\gamma)\). Then, it holds that
\begin{equation} \label{eq:A-Hg-HG-estimate1}
\left| (\HG\psi^e, \bA^e)_{\Gamma_h} - (\Hg \psi, \bA)_{\gamma} \right| \lesssim h^2  \|\psi \|_{H^3(\gamma)}\|\bA\|_{H^1(\gamma)}. 
\end{equation}
Consequently, 
\begin{equation} \label{eq:A-Hg-HG-estimate} 
\left| \bigl(\HG \psi^e - (\Hg \psi)^e, \bA^e \bigr)_{\Gamma_h} \right| \lesssim h^2 \|\psi \|_{H^3(\gamma)} \|\bA\|_{H^1(\gamma)}. 
\end{equation} 
\end{lemma}
\begin{proof}
Since all the Hessian-type tensors involved are symmetric, we may assume without loss of generality that $\bA$ is symmetric. We choose $\tilde{F}(y,\bz) = \mathrm{H}_{\bz} \psi(y) : \bA(y)$, where
\[
\bP_{\bz}=\mathbf{I}-\bz\otimes\bz, \quad 
\bz^\times \bm\xi=\bz\times\bm\xi, \quad 
\mathrm{H}_{\bz}\psi (y) = \mathrm{sym}
\left(\bz^\times\nabla(\bP_{\bz}\nabla \psi^e)^e(y)\bP_{\bz}\right),
\]
and the extension is the closest point projection on $\gamma$.
Comparing with the definition of $\Hg$ in \eqref{eq:H-Hessian}, at $\bz=\bnu(y)$ we immediately have $\mathrm{H}_{\bnu}\psi=\Hg\psi$, whence $F^\bnu(y) = F(y, \bnu(y)) = \mathrm{H}_\gamma \psi(y) : \bA(y)$ for $y \in \gamma$.

Since $\bA$ and $\psi$ are independent of $\bz$, every $\bz$-derivative of $\tilde{F}$ falls on $\bP_{\bz}$ or $\bz^\times$ and leaves the order of differentiation in $\psi$ unchanged. Hence, the quantities $\bq_F$, $q_F^\perp$, and $M_F$ from \eqref{def:geo-qM} are finite sums of contractions of $\bA$, surface derivatives of $\psi$ of order at most two, and smooth coefficients built from $\bnu$ and the Weingarten map $\bH=\nabla_\gamma\bnu$, while $\div_\gamma\bq_F$ carries in addition one derivative of $\bA$, third derivatives of $\psi$, and one derivative of $\bH$. The last group is where the assumption $\gamma\in C^3$ is used: it guarantees $\|\nabla_\gamma\bH\|_{L^\infty(\gamma)}\lesssim 1$. Consequently,
\begin{equation*}
\begin{aligned}
\|F^\bnu\|_{L^1(\gamma)}
+\|\bq_F\|_{L^1(\gamma)}
+\|q_F^\perp\|_{L^1(\gamma)}
+\|M_{F}\|_{L^1(\gamma)}
&\lesssim
 \|\psi\|_{H^2(\gamma)} \|\bA\|_{L^2(\gamma)},\\
\|\div_\gamma\bq_F\|_{L^1(\gamma)}
&\lesssim
 \|\psi\|_{H^3(\gamma)} \|\bA\|_{H^1(\gamma)}.
\end{aligned}
\end{equation*}
A direct application of Theorem \ref{thm:normal-separated} (normal-separated geometric estimate) yields
\begin{equation} \label{eq:geo-stokes-calculus-bound}
\left|(\mathrm{H}_{\bnu_h^\ell}\psi (\bp(x)), \bA^e(x) )_{\Gamma_h} - (\mathrm{H}_\gamma\psi, \bA)_\gamma \right| \lesssim
 h^2 \|\psi\|_{H^3(\gamma)} \|\bA\|_{H^1(\gamma)},
\end{equation}
where, by the definition of $\mathrm{H}_{\bz}\psi$, we have for $x \in \Gamma_h$
\begin{equation} \label{eq:geo-stokes-Hnuh}
\begin{aligned}
\mathrm{H}_{\bnu_h^\ell}\psi (\bp(x)) : \bA^e(x) &= \bnu_h^\ell(\bp(x))^\times \nabla (\bP_{\bnu_h^\ell} \nabla \psi^e)^e(\bp(x)) \bP_{\bnu_h^\ell}(\bp(x)) : \bA^e(x)\\
&= \bnu_h(x)^\times \nabla (\bP_{\bnu_h^\ell} \nabla_\gamma\psi)^e(\bp(x)) \bP_h(x) : \bA^e(x)\\
&= -\nabla (\bP_{\bnu_h^\ell} \nabla_\gamma\psi)^e(\bp(x)) : \bnu_h(x)^\times \bA^e(x) \bP_h(x).
\end{aligned}
\end{equation}
Here, we utilized the identity $\bnu_h^\ell(\bp(x)) = \bnu_h(x)$ on $\Gamma_h$, and the fact that $\nabla\psi^e(y) = \nabla_\gamma\psi(y)$ for any $y \in \gamma$ due to the properties of the closest point projection. 

To compare $(\mathrm{H}_{\bnu_h^\ell}\psi (\bp(x)), \bA^e(x) )_{\Gamma_h}$ in \eqref{eq:geo-stokes-calculus-bound} with $(\mathrm{H}_{\Gamma_h}\psi^e, \bA^e)_{\Gamma_h}$ in \eqref{eq:A-Hg-HG-estimate1}, we first decompose $\mathrm{H}_{\Gamma_h}\psi^e(x) :\bA^e(x)$. By the definition of trace-free Hessian in \eqref{eq:H-Hessian}, $\mathrm{H}_{\Gamma_h}\psi^e(x) = \mathrm{sym}( \bnu_h^\times(x) \nabla(\bP_h \nabla \psi^e)^{e|_{\Gamma_h}}(x) \bP_h(x))$ for $x \in \Gamma_h$, where the outermost extension is taken arbitrarily with respect to $\Gamma_h$. Thus, for $x \in \Gamma_h$, we have
\begin{equation} \label{eq:Stokes-I123}
\begin{aligned}
\mathrm{H}_{\Gamma_h} \psi^e(x) : \bA^e(x) &= \bnu_h(x)^\times \nabla(\bP_h \nabla \psi^e)^{e|_{\Gamma_h}}(x) \bP_h(x) : \bA^e(x) \\
&= -\nabla\big(\bP_h(\mathbf{I} - d \bH)(\nabla_\gamma\psi)^e \big)^{e|_{\Gamma_h}}(x) : 
\bnu_h(x)^\times \bA^e(x) \bP_h(x) \\
&= -\nabla\big(\bP_h (\nabla_\gamma\psi)^e \big)^{e|_{\Gamma_h}}(x) : 
\bnu_h(x)^\times \bA^e(x) \bP_h(x) \\
& \quad + d(x) \nabla\big(\bP_h \bH (\nabla_\gamma\psi)^e \big)^{e|_{\Gamma_h}}(x) : \bnu_h(x)^\times \bA^e(x) \bP_h(x) \\
& \quad + \bP_h(x) \nabla d(x) \otimes \big(\bH (\nabla_\gamma\psi)^e \big)^{e|_{\Gamma_h}}(x) : \bnu_h(x)^\times \bA^e(x) \bP_h(x) \\
&:= I_1 + I_2 + I_3.
\end{aligned}
\end{equation}

{\it \underline{Estimate of $I_1$}: Pointwise high-order approximation of $\mathrm{H}_{\bnu_h^\ell}\psi (\bp(x)) : \bA^e(x)$.}
Observe from \eqref{eq:geo-stokes-Hnuh} that $(\bP_{\bnu_h^\ell} \nabla_\gamma \psi)^e|_{\Gamma_h} = \bP_h (\nabla_\gamma\psi)^e|_{\Gamma_h}$, which allows us to take $\big(\bP_h (\nabla_\gamma\psi)^e \big)^{e|_{\Gamma_h}} = (\bP_{\bnu_h^\ell} \nabla_\gamma \psi)^e$ in the term $I_1$. Consequently, denoting $\bm{w} := \bP_{\bnu_h^\ell} \nabla_\gamma \psi$, we obtain
$$
\begin{aligned}
& \nabla\big(\bP_h (\nabla_\gamma\psi)^e \big)^{e|_{\Gamma_h}}(x) - \nabla (\bP_{\bnu_h^\ell} \nabla_\gamma\psi)^e(\bp(x)) = \nabla \bm{w}^e(x) - \nabla \bm{w}^e(\bp(x)) \\
=~& (\nabla_\gamma \bm{w})^e(x) (\mathbf{I} - d(x) \bH(x)) - (\nabla_\gamma \bm{w})^e(x) = -d(x) (\nabla_\gamma \bm{w})^e(x) \bH(x).
\end{aligned}
$$
Using $|d| \lesssim h^2$ and the norm equivalence, we have 
\begin{equation}\label{eq:Stokes-H-I1}
\begin{aligned}
\left|(\mathrm{H}_{\bnu_h^\ell}\psi (\bp(x)), \bA^e(x) )_{\Gamma_h} - \int_{\Gamma_h} I_1\mathrm{d}\sigma_h \right| 
&\lesssim h^2 \sum_{K \in \mathcal{T}_h} \|\bm{w}\|_{H^1(K^\gamma)} \|\bA\|_{L^2(K^\gamma)} \\
&\lesssim h^2 \|\psi\|_{H^2(\gamma)} \|\bA\|_{L^2(\gamma)}.
\end{aligned}
\end{equation}

{\it \underline{Estimate of $I_2$}: Pointwise high-order bound.} Following a similar argument, we can choose 
$(\bP_h \bH (\nabla_\gamma\psi)^e)^{e|_{\Gamma_h}} = \bP_{\bnu_h^\ell} \bH (\nabla_\gamma\psi)^e$ in $I_2$. 
By utilizing the geometric bound $|d| \lesssim h^2$ and the norm equivalence once again, we directly obtain
\begin{equation*}
\left| \int_{\Gamma_h} I_2 \mathrm{d}\sigma_h \right| \lesssim h^2 \|\psi\|_{H^2(\gamma)} \|\bA\|_{L^2(\gamma)}.
\end{equation*}

{\it \underline{Estimate of $I_3$}: $\bP_h\bnu$-type estimate.} We take $\big(\bH (\nabla_\gamma\psi)^e \big)^{e|_{\Gamma_h}} = \bH (\nabla_\gamma\psi)^e$ in $I_3$. 
Since $|\bnu_h-\bnu^e|, |\bP_h\bnu^e | \lesssim h$, it follows that
\[
\begin{aligned}
\left| \int_{\Gamma_h} I_3 \mathrm{d}\sigma_h \right| & = \left| \int_{\Gamma_h} \bP_h \bnu \cdot (\bnu_h^\times \bA^e \bP_h \bH (\nabla_\gamma \psi)^e ) \right| \\
&\leq   \left| \int_{\Gamma_h} \bP_h \bnu \cdot (\bnu^\times \bA \bH \nabla_\gamma \psi)^e  \right| + C h^2\|\psi\|_{H^1(\gamma)} \|\bA\|_{L^2(\gamma)}.
\end{aligned}
\]
Apply Corollary \ref{cor:Phn} ($\bP_h\bnu$-type estimate) with $\bm{\chi} = \bnu^\times \bA \bH \nabla_\gamma \psi$, then 
\begin{equation} \label{eq:Stokes-H-I3}
\left| \int_{\Gamma_h} I_3 \mathrm{d}\sigma_h \right| \lesssim h^2 \|\bnu^\times \bA \bH \nabla_\gamma \psi \|_{W_1^1(\gamma)} + h^2\|\psi\|_{H^1(\gamma)} \|\bA\|_{L^2(\gamma)}
\lesssim h^2\|\psi\|_{H^2(\gamma)} \|\bA\|_{H^1(\gamma)}.
\end{equation}

Combining the estimates \eqref{eq:Stokes-H-I1}--\eqref{eq:Stokes-H-I3} with \eqref{eq:Stokes-I123}, we obtain
$$
\left|(\mathrm{H}_{\bnu_h^\ell}\psi (\bp(x)), \bA^e(x) )_{\Gamma_h} - (\mathrm{H}_{\Gamma_h} \psi^e(x), \bA^e(x))_{\Gamma_h} \right| \lesssim
 h^2 \|\psi\|_{H^2(\gamma)} \|\bA\|_{H^1(\gamma)}.
$$
Combining this bound with \eqref{eq:geo-stokes-calculus-bound} and the area element estimate $|\mu_h - 1| \lesssim h^2$, we arrive at the desired estimates \eqref{eq:A-Hg-HG-estimate1} and \eqref{eq:A-Hg-HG-estimate}.
\end{proof}

\begin{lemma}[surface curl--divergence consistency]\label{lem:stokes-curl-div}
Let $\gamma$ be a $C^3$ surface, let $\bA$ be a symmetric tangential second-order tensor on \(\gamma\), with components in \(H^1(\gamma)\), and let \(\psi\in H^3(\gamma)\).
Then, it holds that
\begin{equation}
\label{eq:stokes-curl-div}
\left|
(\curl_{\Gamma_h}\psi^e, \mathrm{div}_{\Gamma_h}\bA^e)_{\Gamma_h} - (\curl_\gamma \psi, \mathrm{div}_\gamma \bA)_\gamma \right|
\lesssim
h^2  \|\psi\|_{H^2(\gamma)} \|\bA\|_{H^1(\gamma)}.
\end{equation}
\end{lemma}

\begin{proof}
We choose $\tilde{F}(y,\bz) = \curl_{\bz}\psi(y) \cdot {\rm div}_{\bz}\bA(y)$, where
\[
\begin{aligned}
\curl_{\bz}\psi &= \bz^\times\nabla\psi^e,\quad \bz^\times \bm\xi=\bz\times\bm\xi, \\
(\div_{\bz}\bA)_i &= \tr\!\left(\bP_{\bz}\nabla \bA_i^e\right),  \quad \bP_{\bz}=\mathbf{I}-\bz\otimes\bz.
\end{aligned}
\]
Here, $\bA_i^e$ denotes the $i$-th row of $\bA^e$.   Since \(\bA\) and \(\psi\) do not depend on \(\bz\), every $\bz$-derivative of $\tilde{F}$ falls on $\bP_{\bz}$ or $\bz^\times$ and leaves the order of differentiation in $\psi$ unchanged. Hence, it is straightforward that the quantities defined in \eqref{def:geo-qM} satisfy 
\begin{equation}\label{eq:stokes-calculus-bound}
\|F^\bnu\|_{L^1(\gamma)}
+\|\bq_F\|_{L^1(\gamma)}
+\|q_F^\perp\|_{L^1(\gamma)}
+\|M_{F}\|_{L^1(\gamma)}
\lesssim
 \|\psi\|_{H^1(\gamma)} \|\bA\|_{H^1(\gamma)}.
\end{equation}

Next, we provide a more precise characterization of the tangential vector field $\bq_F(y) := \bP(y) D_{\bz} \tilde{F}(y, \bnu(y))$. For any tangent vector $\bm\xi \perp \bnu$, the component of $\bq_F$ along $\bm\xi$ can be decomposed as $\bq_F \cdot \bm\xi = \bq_1 \cdot \bm\xi + \bq_2 \cdot \bm\xi$,
where
\begin{equation}\label{eq:stokes-curl-div-q12}
\begin{aligned}
\bq_1(y) \cdot\bm\xi
&:=
D_\bz\curl_{\bz}\psi(y)\big|_{\bz=\bnu}
[\bm\xi]\cdot\divg \bA(y),
\\
\bq_2(y) \cdot\bm\xi
&:=
\curl_\gamma\psi\cdot
D_\bz
\div_{\bz}\bA(y)\big|_{\bz=\bnu}
[\bm\xi].
\end{aligned}
\end{equation}
For $\bq_1$ on $\gamma$, note that $\curl_{\bz}\psi=\bz^\times\nabla\psi^e$ is linear in $\bz$, so
$D_\bz\curl_{\bz}\psi \big|_{\bz=\bnu}[\bm\xi]=\bm\xi\times\nabla_\gamma\psi$ and hence
$\bm{q}_1 \cdot\bm\xi = (\bm\xi\times\nabla_\gamma\psi) \cdot \divg \bA$.
Since $\bm\xi$ and $\nabla_\gamma\psi$ are both tangential, their cross product satisfies $\bm\xi\times\nabla_\gamma\psi
 = -(\bm\xi\cdot\curl_\gamma\psi)\,\bnu$.
Moreover, differentiating $\bnu^\top\bA=0$ yields the surface identity $\divg\bA\cdot\bnu=-\bA:\nabla_\gamma\bnu$.
Combining these two facts,
\[
\begin{aligned}
    \bm{q}_1 \cdot\bm\xi
    &= -(\bm\xi\cdot\curl_\gamma\psi)(\divg \bA\cdot\bnu)
    =(\bm\xi\cdot\curl_\gamma\psi)(\bA:\nabla_\gamma\bnu).
\end{aligned}
\]
For $\bq_2$ on $\gamma$, since $(\div_{\bz}\bA)_i=\tr(\bP_{\bz}\nabla\bA_i^e)$ with $D_\bz\bP_{\bz}[\bm\xi]=-(\bm\xi\otimes\bz+\bz\otimes\bm\xi)$, we have
\[
D_\bz(\div_{\bz}\bA)_i\big|_{\bz=\bnu}[\bm\xi]
=
-\tr\!\big(
(\bm\xi\otimes\bnu+\bnu\otimes\bm\xi)
\nabla \bA_i^e
\big)
=
-\bnu\cdot(\nabla \bA_i^e)\bm\xi - \bm\xi\cdot(\nabla \bA_i^e)\bnu .
\]
The second term vanishes because $\bA_i^e$ is the closest-point extension, so $(\nabla \bA_i^e)\bnu=0$; and differentiating $\bA_i\cdot\bnu=0$ along $\bm\xi$ gives $-\partial_\xi \bA_i^e\cdot\bnu=\bA_i^e\cdot\partial_\xi\bnu$. Hence $D_\bz(\div_{\bz}\bA)_i\big|_{\bz=\bnu}[\bm\xi]=\bA_i^e\cdot\partial_\xi\bnu$, that is, $D_\bz\div_{\bz}\bA\big|_{\bz=\bnu}[\bm\xi]=\bA \nabla_\gamma\bnu \bm\xi$, and therefore
\[
    \bm{q}_2 \cdot\bm\xi = \curl_\gamma\psi\cdot\bA\nabla_\gamma\bnu\,\bm\xi .
\]

Substituting the expressions for $\bm{q}_i \cdot\bm\xi$ $(i=1,2)$ into \eqref{eq:stokes-curl-div-q12}, and recalling that $\nabla_\gamma\bnu=\bH$ is the (symmetric) Weingarten map, we obtain
\begin{equation}\label{eq:stokes-qF}
\bq_F = (\bA:\bH)\,\curl_\gamma\psi + \bH\bA\,\curl_\gamma \psi
= \big[(\bA:\bH)\,\mathbf{I} + \bH\bA\big]\curl_\gamma\psi.
\end{equation}
Thus $\bq_F$ is a finite sum of contractions of $\bA$ with $\curl_\gamma\psi=\bnu\times\nabla_\gamma\psi$ and the curvature coefficients in $\bH$. Its surface divergence therefore involves at most $\nabla_\gamma\bA$ paired with $\nabla_\gamma\psi$, $\bA$ paired with $\nabla_\gamma^2\psi$, and $\nabla_\gamma\bH$ paired with $\bA$ and $\nabla_\gamma\psi$. The last group is where the assumption $\gamma\in C^3$ is used: it guarantees $\|\nabla_\gamma\bH\|_{L^\infty(\gamma)} \lesssim 1$. Consequently,
\begin{equation}\label{eq:stokes-div-q-low-order}
        \|\divg\bq_F\|_{L^1(\gamma)}
        \lesssim
        \|\psi\|_{H^2(\gamma)}         \|\bA\|_{H^1(\gamma)}.
\end{equation}

Thanks to \eqref{eq:stokes-calculus-bound} and \eqref{eq:stokes-div-q-low-order}, a direct application of Theorem \ref{thm:normal-separated} (normal-separated geometric estimate) gives 
\begin{equation} \label{eq:stokes-curl-div-1}
\left| \big( \curl_{\bnu_h^\ell} \psi (\bp(x)), \div_{\bnu_h^\ell}\bA(\bp(x)) \big)_{\Gamma_h} - (\curl_\gamma \psi, \div_{\gamma} \bA)_\gamma \right| \lesssim
 h^2 \|\psi\|_{H^2(\gamma)} \|\bA\|_{H^1(\gamma)},
\end{equation}
where 
$$
\begin{aligned}
\curl_{\bnu_h^\ell} \psi (\bp(x)) = \bnu_h^\times \nabla \psi^e (\bp(x)), \quad (\div_{\bnu_h^\ell}\bA)_i (\bp(x))  = \tr\!\left(\bP_{h} \nabla \bA_i^e (\bp(x)) \right).
\end{aligned}
$$ 

It remains to replace the operators evaluated at the projected point $y = \bp(x)$ by the discrete surface operators evaluated at $x \in \Gamma_h$. Recall that $\curl_{\Gamma_h}\psi^e(x) = \bnu_h(x)^\times \nabla\psi^e(x)$ and $(\div_{\Gamma_h}\bA^e)_i(x) = \tr(\bP_h(x) \nabla\bA_i^e(x))$; these differ from $\curl_{\bnu_h^\ell}\psi(\bp(x))$ and $(\div_{\bnu_h^\ell}\bA)_i(\bp(x))$ only in that the ambient gradient is evaluated at $x$ rather than at $\bp(x)$, since the normal factors coincide, $\bnu_h^\ell(\bp(x)) = \bnu_h(x)$ and $\bP_{\bnu_h^\ell}(\bp(x)) = \bP_h(x)$. By \eqref{eq:nabla-trans}, for any $v \in H^1(\gamma)$,
\[
\nabla v^e(x) - \nabla v^e(\bp(x))
= (\bP - d\bH)(\nabla_\gamma v)^e(x) - (\nabla_\gamma v)^e(x)
= -d(x)\bH (\nabla_\gamma v)^e(x),
\]
so that $|d| \lesssim h^2$ yields the pointwise bounds
\[
\begin{aligned}
\left|\curl_{\Gamma_h}\psi^e(x) - \curl_{\bnu_h^\ell}\psi(\bp(x))\right| &\lesssim h^2 |(\nabla_\gamma\psi)^e(x)|, \\
\left|\div_{\Gamma_h}\bA^e(x) - \div_{\bnu_h^\ell}\bA(\bp(x))\right| &\lesssim h^2 |(\nabla_\gamma\bA)^e(x)|.
\end{aligned}
\]
Then, the Cauchy--Schwarz inequality together with the norm equivalence \eqref{eq:scalar-norm-equiv} gives
\[
\left|(\curl_{\Gamma_h}\psi^e, \div_{\Gamma_h}\bA^e)_{\Gamma_h}
- \big(\curl_{\bnu_h^\ell}\psi(\bp(x)), \div_{\bnu_h^\ell}\bA(\bp(x))\big)_{\Gamma_h}\right|
\lesssim h^2 \|\psi\|_{H^1(\gamma)} \|\bA\|_{H^1(\gamma)}.
\]
Combining this with \eqref{eq:stokes-curl-div-1} yields the desired estimate \eqref{eq:stokes-curl-div}.
\end{proof}

\begin{corollary}[Stokes-type tensor Green consistency] \label{cor:stokes-tensor-green}
Let $\gamma$ be a $C^3$ surface, let $\bA$ be a symmetric tangential second-order tensor on \(\gamma\), with components in \(H^1(\gamma)\), and let \(\psi\in H^3(\gamma)\). Then,
\begin{equation} \label{eq:jh-extension}
\left| (\HG\psi^e, \bA^e)_{\Gamma_h} + (\curl_{\Gamma_h}\psi^e, \mathrm{div}_{\Gamma_h}\bA^e)_{\Gamma_h} \right|
\lesssim h^2 \|\psi\|_{H^3(\gamma)} \|\bA\|_{H^1(\gamma)}.
\end{equation}
\end{corollary}
\begin{proof}
Adding the estimates \eqref{eq:A-Hg-HG-estimate1} and \eqref{eq:stokes-curl-div} from Lemma~\ref{lem:neilan-volume-estimate} and Lemma~\ref{lem:stokes-curl-div} gives
\[
\begin{aligned}
& \quad (\HG\psi^e, \bA^e)_{\Gamma_h} + (\curl_{\Gamma_h}\psi^e, \div_{\Gamma_h}\bA^e)_{\Gamma_h} \\
 & =   (\Hg\psi, \bA)_\gamma + (\curl_\gamma\psi, \div_\gamma\bA)_\gamma
+ \mathcal{O}\!\left(h^2 \|\psi\|_{H^3(\gamma)} \|\bA\|_{H^1(\gamma)}\right).
\end{aligned}
\]
The two continuous inner products cancel due to integration by parts on the closed surface $\gamma$, which proves \eqref{eq:jh-extension}.
\end{proof}

\subsection{Applications to tangential tensors: boundary-flux consistency}\label{subsec:geo-boundary}

We finally turn to the skeleton counterpart of the interior consistency of
Section~\ref{subsec:geo-tensor}. The deliverable is the conormal-flux
consistency of Corollary~\ref{cor:conormal-flux}, used in the error
analysis of Section~\ref{sec:analysis}. It rests on a single auxiliary
tool---a weak conormal-flux estimate for a tangential vector field---to
which we turn first.

\begin{lemma}[weak conormal-flux consistency]
\label{lem:weak-conormal-flux}
Let $\gamma$ be a $C^3$ surface,
let \(\bm\chi\) be a tangential vector field on \(\gamma\), with
\(\bm\chi\in \bm{W}_1^1(\gamma)\). Then
\[
\left|
\sum_{K\in\mathcal T_h}
\int_{\partial K}
\bm\chi^e\cdot\bm n_K\,\dif s
\right|
\lesssim
h^2
\|\bm\chi\|_{W_1^1(\gamma)} .
\]
\end{lemma}

\begin{proof}
The elementwise Green's formula gives
\[
\sum_{K\in\mathcal T_h}
\int_{\partial K}
\bm\chi^e\cdot\bm n_K\,\dif s
=
\int_{\Gamma_h}
\div_{\Gamma_h}\bm\chi^e\,\dif\sigma_h .
\]
We then set $\tilde{F}(y,\bz) := \tr(\bP_{\bz}\nabla\bm\chi^e(y))$, which satisfies
\[
    F(y,\bnu)=\div_\gamma\bm\chi,
        \quad
        F(y,\bnu_h^\ell(y))|_{ y = \bp(x)}
        =
        \div_{\Gamma_h}\bm\chi^e(x)
        + \mathcal{O}(h^2|(\nabla_\gamma\bm\chi)^e(x)|).
\]
Since \(\gamma\) is closed and
\(\bm\chi\) is tangential, it holds that $\int_\gamma F(y,\bnu)\dif\sigma =
\int_\gamma\div_\gamma\bm\chi\dif\sigma = 0$.

Following the same computation as in Lemma~\ref{lem:stokes-curl-div}, we determine the tangential vector field $\bq_F(y) := \bP(y) D_{\bz} \tilde{F}(y, \bnu(y))$ explicitly. Indeed, for every $\bm{\xi}\perp\bnu$,
\[
D_\bz F(y,\bnu)[\bm{\xi}]
= \tr\!\left(
-(\bm{\xi}\otimes\bnu+\bnu\otimes\bm{\xi})\nabla {\bm\chi}^e
\right)
=
-\bnu\cdot(\nabla {\bm\chi}^e)\bm{\xi},
\]
where we use \((\nabla {\bm\chi}^e)\bnu=0\) since \({\bm\chi}^e\) is the closest-point extension.
Moreover, \({\bm\chi}\cdot\bnu=0\) on \(\gamma\), and hence
\[
        \bnu\cdot(\nabla {\bm\chi}^e)\bm{\xi}
        =
        \partial_\xi {\bm\chi}\cdot\bnu
        =
        -{\bm\chi}\cdot\partial_\xi\bnu
        =
        -{\bm\chi}\cdot\bH\bm{\xi}.
\]
Combining the two displays gives $D_\bz F(y,\bnu)[\bm{\xi}] = {\bm\chi}\cdot\bH\bm{\xi}$, so that $\bq_F = \bH\bm{\chi}$ and $q_F^\perp = 0$. Consequently, $\div_\gamma\bq_F = \div_\gamma(\bH\bm\chi)$ involves one derivative of the Weingarten map $\bH$; the assumption $\gamma\in C^3$ guarantees $\|\nabla_\gamma\bH\|_{L^\infty(\gamma)} \lesssim 1$, whence
\[
\|F^\bnu\|_{L^1(\gamma)} + \|{\rm div}_\gamma\bq_F\|_{L^1(\gamma)} + h\|\bq_F\|_{L^1(\gamma)} + \|q_F^\perp\|_{L^1(\gamma)} + \|M_F\|_{L^1(\gamma)}
\lesssim \|\bm\chi\|_{W_1^1(\gamma)} .
\]
Theorem~\ref{thm:normal-separated} (normal-separated geometric estimate) then gives the desired estimate.
\end{proof}

The single estimate needed in the error analysis is the following specialization.

\begin{corollary}[conormal-flux consistency]\label{cor:conormal-flux}
Let $\gamma$ be a $C^3$ surface,
let \(\bA\) be a symmetric tangential two-tensor on \(\gamma\), whose
components belong to \(H^1(\gamma)\), and let \(\psi\in H^2(\gamma)\).  Then
\begin{equation}\label{eq:geo-trace-2}
\left|
\langle(\curl_\gamma\psi)^e, \llbracket \bA^e\bm n_h\rrbracket\rangle_{\mathcal E_h}
\right|
\lesssim
h^2 \|\psi\|_{H^2(\gamma)} \|\bA\|_{H^1(\gamma)}.
\end{equation}
\end{corollary}
\begin{proof}
Set \(\bm\chi =\bA \curl_\gamma\psi\), a tangential field satisfying 
\(\|\bm\chi\|_{W_1^1(\gamma)}\lesssim \|\psi\|_{H^2(\gamma)} \|\bA\|_{H^1(\gamma)}\ \).
Since \(\bA^e\) is single-valued along each interior edge,
\(\bm\chi^e\cdot\bm n_K=(\curl_\gamma\psi)^e\cdot\bA^e\bm n_K\), whence
\[
\sum_{K\in\mathcal T_h}\int_{\partial K}\bm\chi^e\cdot\bm n_K\,\dif s
= \langle(\curl_\gamma\psi)^e, \llbracket \bA^e\bm n_h\rrbracket\rangle_{\mathcal E_h}.
\]
The claim now follows from Lemma~\ref{lem:weak-conormal-flux} (weak
conormal-flux estimate).
\end{proof}

\section{Finite element method} \label{sec:FEM}
This section introduces the Morley finite element on the discrete surface $\Gamma_h$ and the corresponding stabilized discrete formulation. Both the finite element space and the discrete bilinear form are intrinsic to $\Gamma_h$; the surface geometry enters only through the error analysis of Section~\ref{sec:analysis}, which relies essentially on the geometric consistency estimates of Section~\ref{sec:geo-error} to control the normal- and conormal-dependent errors produced by the discretization.

\subsection{Morley element on surfaces}
All degrees of freedom of the Morley element---the vertex values and the edge conormal-derivative moments---are scalar-valued. Assembling them across elements therefore requires neither identifying the (generally distinct) tangent planes of adjacent triangles nor any global continuity of the finite element function. This allows the Morley finite element space to be defined naturally on the piecewise linear surface $\Gamma_h$ as
\begin{equation*}
\begin{aligned}
V_h := \big\{ w\in L^2(\Gamma_h) \,\big|\, w|_K \in \mathcal{P}_2(K),\  
 w(a) \text{ is single-valued}\quad &\forall\, a \in \mathcal{V}_h, \\
 \int_e [\nabla_{\Gamma_h} w \cdot \bm{n}_h]\, \mathrm{d}s_h = 0 
\quad  &\forall\, e \in \mathcal{E}_h \big\}.
\end{aligned}
\end{equation*}
Let $\mathring{V}_h$ denote the subspace of $V_h$ consisting of functions with zero mean over $\Gamma_h$.

As in the planar case, functions in $V_h$ satisfy the following weak continuity conditions:
\begin{equation}\label{eq:weak-continuity}
\int_e [\nabla_{\Gamma_h} w \cdot \bm{t}_h]\, \mathrm{d}s_h = 0,
\quad
\int_e [\nabla_{\Gamma_h} w \cdot \bm{n}_h]\, \mathrm{d}s_h = 0
\quad \forall w \in V_h,\ \forall e \in \mathcal{E}_h .
\end{equation}
It should be noted that the co-normal vector $\bm{n}_h$ is defined elementwise and generally differs on the two elements adjacent to an edge $e$. Consequently, the vector-valued jump $\int_e \llbracket \nabla_{\Gamma_h} w \rrbracket \, \mathrm{d}s_h$ does not vanish in general.

\paragraph{$H^1$-conforming relative}
Define the operator
\(
\Pi_h^c : V_h \to C^0(\Gamma_h)
\)
such that $\Pi_h^c w$ is the continuous piecewise quadratic Lagrange finite element function whose degrees of freedom are prescribed by
\begin{equation}\label{eq:Pi_h^c}
(\Pi_h^c w)(a) = w(a) \quad \forall a \in \mathcal{V}_h,
\qquad
(\Pi_h^c w)(a_e) = \{w\}(a_e) \quad \forall a_e \in \mathcal{M}_h,
\end{equation}
where $\mathcal{M}_h$ denotes the set of edge midpoints of $\mathcal{T}_h$.
By a standard scaling argument and norm equivalence on the reference element, one obtains
\begin{equation}\label{eq:sNZT-relative-1}
\| w - \Pi_h^c w \|_{H^m(K)}
\lesssim
h^{1/2-m} \| [w] \|_{L^2(\partial K)}
\qquad
\forall w \in V_h,\ \forall K \in \mathcal{T}_h,\ 0 \le m \le 2.
\end{equation}

\subsection{The surface Piola transform}
The surface Piola transform is introduced to define the finite element interpolation in the next subsection. Because the Morley degrees of freedom include edge conormal-derivative moments, interpolating a function on $\gamma$ requires transferring its surface gradient to $\Gamma_h$ through a map that preserves conormal components~(cf.~\cite{demlow2024tangential,demlow2025taylor}).

Let $\Phi: \mathscr{S}_0 \to \mathscr{S}_1$ be a diffeomorphism between surfaces. The Piola transform $\mathscr{P}_{\Phi}$ maps a tangential vector field $\bm{g}$ on $\mathscr{S}_0$ to $\mathscr{S}_1$ via
$$
(\mathscr{P}_{\Phi}\bm{g}) \circ \Phi := |D\Phi|^{-1} (D\Phi) \bm{g},
$$
where $\mu := |D\Phi|$ satisfies $\mu \mathrm{d}\sigma_0 = \mathrm{d}\sigma_1$ for the respective surface measures. The inverse transform for $\bm{g}: \mathscr{S}_1 \to \mathbb{R}^3$ is given by 
$$
(\mathscr{P}_{\Phi^{-1}} \bm{g}) \circ \Phi^{-1} := (\mu \circ \Phi^{-1}) (D\Phi^{-1}) \bm{g}.
$$
This transformation preserves the divergence:
\begin{equation*}
{\rm div}_{\mathscr{S}_0} \bm{g} = \mu {\rm div}_{\mathscr{S}_1} \mathscr{P}_{\Phi} \bm{g} \qquad \forall \bm{g} \in \bm{H}({\rm div}_{\mathscr{S}_0}; \mathscr{S}_0).
\end{equation*}
Applying to the surface bijection $\bm{p}: \Gamma_h \to \gamma$ (where $\mu = \mu_h$), the Piola pullback of $\bm{g}: \gamma \to \mathbb{R}^3$ to $\Gamma_h$ is explicitly formulated as:
\begin{equation}\label{def:Piola_invp}
\breve{\bm{g}} := \mathscr{P}_{\bm{p}^{-1}}\bm{g} = \mu_h\Big[\textbf{I}-\frac{\bm{\nu}\otimes \bm{\nu}_h}{\bm{\nu}\cdot \bm{\nu}_h} \Big][\textbf{I}-d\textbf{H}]^{-1}\bm{g}^e.
\end{equation}

We first state the norm equivalence for this transform (cf.~\cite[Lemma~4.1]{bonito2020divergence}, \cite[Lemma~2.2]{wu2025stabilized}).
\begin{lemma}[norm equivalence of the Piola transform] \label{lm:Piola-equivalence}
For any $K \in \mathcal{T}_h$, let $\bm{g} \in \bm{H}^m(K^\gamma)$ with $m \in \{0, 1, 2\}$ on a $C^{m+2}$ surface. Then, $\breve{\bm g} = \mathscr{P}_{\bm{p}^{-1}} \bm{g} \in \bm{H}^m(K)$ satisfies 
\begin{equation} \label{eq:Piola-equivalence}
\|\bm{g}\|_{H^m(K^\gamma)} \simeq \|\breve{\bm g}\|_{H^m(K)}.
\end{equation}
\end{lemma}

The following lemma (cf.~\cite[Lemma~2.3]{wu2025stabilized}) and its corollary provide pointwise and elementwise approximation estimates for the Piola transform.
\begin{lemma}[pointwise Piola estimate for the gradient]\label{lm:Piola-derivative}
For any $w \in C^1(\gamma)$, it holds on $\Gamma_h$ that
\begin{equation}\label{eq:Piola-derivative}
\big| \widebreve{\nabla_\gamma w} - \nabla_{\Gamma_h} w^e \big| \lesssim h^2 \big| (\nabla_\gamma w)^e \big|.
\end{equation}
\end{lemma}

\begin{corollary}\label{cor:Piola-derivative}
For $w \in C^1(\gamma) \cap H_h^3(\gamma)$ on a $C^{4}$ surface, there holds
\begin{equation}\label{eq:Piola-derivative-1}
\| \widebreve{\nabla_\gamma w} - \nabla_{\Gamma_h} w^e \|_{H^m(K)} \lesssim h^{2-m} \| w \|_{H^3(K^\gamma)} \quad \forall K \in \mathcal{T}_h,\ 0 \le m \le 2.
\end{equation}
\end{corollary}
The elementwise estimate \eqref{eq:Piola-derivative-1} follows from Lemma~\ref{lm:Piola-derivative}, the norm equivalence \eqref{eq:Piola-equivalence}, and standard interpolation estimates; the argument is contained in the proof of \cite[Lemma~3.4]{wu2025stabilized} and is therefore omitted.
An analogous estimate holds for the surface curl, which is used to
compare the reconstructed velocity with its discrete counterpart in
Section~\ref{sec:analysis}, whose proof is deferred to Appendix~\ref{sec:appendix_piola_curl}.
\begin{corollary}[pointwise Piola estimate for the curl]\label{cor:Piola-curl}
For any $w \in C^1(\gamma)$, it holds on $\Gamma_h$ that
\begin{equation}\label{eq:Piola-curl}
\big| \widebreve{\curl_\gamma w} - \curl_{\Gamma_h} w^e \big|
\lesssim h^2 \big| (\nabla_\gamma w)^e \big|.
\end{equation}
\end{corollary}

\subsection{Interpolation and approximation}
Given $w \in C^1(\gamma) \cap H_h^3(\gamma)$, the interpolant $\tilde{I}_h w \in V_h$ is defined by prescribing the Morley degrees of freedom. 
At each vertex $a \in \mathcal{V}_h$, the nodal value is given by
\(
\tilde{I}_h w(a) = w^e(a) .
\)
For each edge $e \in \mathcal{E}_h$ and each element $K$ sharing $e$, the normal derivative degree of freedom is defined by
\begin{equation}\label{eq:tildeI}
\int_e \nabla_{\Gamma_h} \tilde{I}_h w|_{K} \cdot \bm{n}_K \, \mathrm{d}s_h
=
\int_e \widebreve{\nabla_\gamma w}|_{K} \cdot \bm{n}_K \, \mathrm{d}s_h .
\end{equation}
Here, $\bm{n}_K$ denotes the conormal vector of $K$ along $e$.
The normal-preserving property of the surface Piola transform ensures that the degrees of freedom are well-defined, and $\tilde{I}_h w$ is uniquely determined as an element of $V_h$.

The integral-free subspace projection $I_h w \in \mathring{V}_{h}$ is then defined as
\begin{equation*}
I_h w := \tilde{I}_h w - \frac{1}{|\Gamma_h|}\int_{\Gamma_h} \tilde{I}_h w \, \mathrm{d}\sigma_h.
\end{equation*}

\begin{lemma}[interpolation approximation] \label{lm:approximation}
For $w \in C^1(\gamma) \cap H_h^3(\gamma)$, it holds that 
\begin{subequations}
\begin{align} 
\| w^e - \tilde{I}_h w \|_{H^m(K)} &\lesssim h^{3-m} \|w\|_{H^3({K^\gamma})} 
\quad \forall K \in \mathcal{T}_h, ~0\leq m\leq 3, \label{eq:tildeI-L0} \\
\| \widebreve{\nabla_\gamma w} - \nabla_{\Gamma_h} I_h w \|_{H^m(K)} &\lesssim h^{2-m} \| w \|_{H^3({K^\gamma})} 
\quad \forall K \in \mathcal{T}_h, ~0\leq m\leq 2. \label{eq:I-D1}
\end{align}
Moreover, if $w \in C^1(\gamma) \cap H_h^3(\gamma) \cap \mathring{L}^2(\gamma)$, it holds that
\begin{equation} \label{eq:I-L0}
\|w^e - I_h w\|_{H_h^m(\Gamma_h)} \lesssim h^{\min\{3-m,2\}} \|w\|_{H_h^3(\gamma)} \quad 0 \leq m \leq 3.
\end{equation}
\end{subequations}
\end{lemma}
\begin{proof}
Fix an element $K \in \mathcal{T}_h$.  
Let $z \in \mathcal{P}_2(K)$ be the elementwise Morley interpolant of $w^e$, defined by 
$z(a) = w^e(a)$ and $\int_e \nabla_{\Gamma_h} z \cdot \bm{n}_K \, \mathrm{d}s_h
=
\int_e \nabla_{\Gamma_h} w^e|_K \cdot \bm{n}_K \, \mathrm{d}s_h$ for all $a \in \mathcal{V}_K$ and $e \in \mathcal{E}_K$ .
By standard approximation theory for quadratic polynomials,
\begin{equation}\label{eq:we-z}
\| w^e - z \|_{H^m(K)} 
\lesssim h^{3-m} | w^e |_{H^3(K)}
\lesssim h^{3-m} \| w \|_{H^3(K^\gamma)}
\quad 0 \le m \le 3 .
\end{equation}

Using the Piola transformation property \eqref{eq:Piola-derivative} and the definition of $\tilde{I}_h$ in \eqref{eq:tildeI}, we obtain for any edge $e \in \mathcal{E}_K$
\[
\begin{aligned}
\left| \int_e \nabla_{\Gamma_h} ( z - \tilde{I}_h w|_K ) \cdot \bm{n}_K \, \mathrm{d}s_h \right|
&=
\left| \int_e \big( \nabla_{\Gamma_h} w^e - \widebreve{\nabla_\gamma w} \big)|_K \cdot \bm{n}_K \, \mathrm{d}s_h \right| \\
&\lesssim h^{5/2} \| \nabla_\gamma w \|_{L^2(e^\gamma)}
\lesssim h^{2} \| w \|_{H^2(K^\gamma)},
\end{aligned}
\]
where the last step follows from a trace inequality.
Since $z - \tilde{I}_h w|_K \in \mathcal{P}_2(K)$ and vanishes at all vertices of $K$, a standard scaling argument and the above inequality yield
\begin{equation}\label{eq:z-Ihw}
\| z - \tilde{I}_h w \|_{H^m(K)}
\lesssim
h^{1-m} \max_{e \in \mathcal{E}_K}
\left| \int_e \nabla_{\Gamma_h} ( z - \tilde{I}_h w|_K ) \cdot \bm{n}_K \, \mathrm{d}s_h \right|
\lesssim
h^{3-m} \| w \|_{H^2(K^\gamma)} .
\end{equation}
Combining \eqref{eq:we-z} and \eqref{eq:z-Ihw} gives \eqref{eq:tildeI-L0}.  
Estimate \eqref{eq:I-D1} then follows from \eqref{eq:Piola-derivative-1} and \eqref{eq:tildeI-L0}.

Finally, if $w \in \mathring{L}^2(\gamma)$, using $|1-\mu_h| \lesssim h^2$, we have
\[
\begin{aligned}
\| I_h w - \tilde{I}_h w \|_{L^2(\Gamma_h)}
&= \frac{1}{|\Gamma_h|^{1/2}} \left| \int_{\Gamma_h} \tilde{I}_h w \, \mathrm{d}\sigma_h \right| \\
&\lesssim
\left| \int_{\Gamma_h} ( \tilde{I}_h w - w^e ) \, \mathrm{d}\sigma_h \right|
+
\left| \int_{\Gamma_h} w^e \, \mathrm{d}\sigma_h - \int_\gamma w \, \mathrm{d}\sigma \right| \\
&\lesssim
\| \tilde{I}_h w - w^e \|_{L^2(\Gamma_h)}
+
\int_{\Gamma_h} | w^e (1-\mu_h) | \, \mathrm{d}\sigma_h \\
&\lesssim
h^3 \| w \|_{H_h^3(\gamma)} + h^2 \| w \|_{L^2(\gamma)} ,
\end{aligned}
\]
which proves \eqref{eq:I-L0}.
\end{proof}

\begin{corollary}[jump estimates of the interpolant]
For $w \in C^1(\gamma) \cap H_h^3(\gamma)$, the following estimates hold:
\begin{subequations}
\begin{align}
\| [I_h w ] \|_{L^2(e)}
&\lesssim h^{5/2} \|w\|_{H_h^3(\omega_{e^\gamma})}
\quad \forall e \in \mathcal{E}_h,  \label{eq:value-jump} \\
\| [\nabla_{\Gamma_h} I_h w \cdot \bm{n}_h] \|_{L^2(e)}
&\lesssim h^{3/2} \|w\|_{H_h^3(\omega_{e^\gamma})}
\quad \forall e \in \mathcal{E}_h, \label{eq:normal-jump} \\
\| [\nabla_{\Gamma_h} I_h w \cdot \bm{t}_h] \|_{L^2(e)}
&\lesssim h^{3/2} \|w\|_{H_h^3(\omega_{e^\gamma})}
\quad \forall e \in \mathcal{E}_h. \label{eq:t-jump}
\end{align}
\end{subequations}
\end{corollary}

\begin{proof}
Since $\tilde{I}_h w$ and $I_h w$ differ only by a constant and $w^e \in C^0(\Gamma_h)$, it holds that $[I_h w]|_e=[\tilde{I}_h w - w^e]|_e$.
Combining the trace inequality and \eqref{eq:tildeI-L0} yields \eqref{eq:value-jump}.  
The estimate \eqref{eq:t-jump} then follows from an inverse inequality on $e$.
Using
\(
[\widebreve{\nabla_\gamma w} \cdot \bm{n}_h]|_e = 0
\)
and \eqref{eq:I-D1}, another application of the trace inequality yields \eqref{eq:normal-jump}.
\end{proof}

\subsection{Stabilized nonconforming FEM}
In this subsection, a stabilized nonconforming finite element method is introduced for the surface problem \eqref{eq:s-streamStokes}.
Define the bilinear forms $a_h : H_h^2(\Gamma_h) \times H_h^2(\Gamma_h) \to \mathbb{R}$ and $s_h^0 :V_h \times V_h  \to \mathbb{R}$ by
\begin{equation*}
\begin{aligned}
a_h(w,\psi)&:=(\HG w, \HG \psi)_{\Gamma_h} +(\nabla_{\Gamma_h} w, \nabla_{\Gamma_h} \psi)_{\Gamma_h}, \\
s_h^0(w,\psi)&:= \sum_{e \in \mathcal{E}_h} h^{-3}_e\langle [w ], [\psi ]\rangle_{e}.
\end{aligned}
\end{equation*}
The stabilization term involves no user-defined parameter.
The stabilized nonconforming finite element method reads: Find $\phi_h \in \mathring{V}_h$ such that
\begin{equation}\label{eq:FEM}
    a_h(\phi_h,\psi)+s_h^0(\phi_h,\psi)=l_h(\psi):=(\bm{f}_h, \curl_{\Gamma_h} \psi)_{\Gamma_h} \quad \forall \psi\in \mathring{V}_{h},
\end{equation}
where $\bm{f}_h$ is an approximation of $\bm{f}$ on $\Gamma_h$. For instance, we can take the extension $\bm{f}_h = \bm{f}^e$.

Define the discrete energy seminorm on $V_h$ by
\begin{equation*}
\triplenorm{w}^2 := \| \HG w \|_{L^2(\Gamma_h)}^2 + \| \nabla_{\Gamma_h} w \|_{L^2(\Gamma_h)}^2 + \sum_{e\in\mathcal E_h}h_e^{-3}\|[w]\|_{L^2(e)}^2.
\end{equation*}
On the full space $V_h$ this is only a seminorm. We show below that it controls the broken $H^2$ norm on the subspace $\mathring{V}_h$ by a discrete Poincar\'e inequality and a discrete Korn's inequality, so that $\triplenorm{\cdot}$ is in fact a norm on $\mathring{V}_h$, ensuring the well-posedness of the discrete problem \eqref{eq:FEM} by the Lax--Milgram lemma.

\begin{lemma}[discrete Poincar\'e inequality]
It holds that
\begin{align}\label{eq:poincare}
\|w\|_{L^2(\Gamma_h)}
\lesssim
|w|_{H_h^1(\Gamma_h)}
\qquad  \forall w \in \mathring{V}_h.   
\end{align}
\end{lemma}
\begin{proof}
Let $w_c \in C^0(\Gamma_h)$ be the continuous piecewise linear vertex interpolant of $w$. A standard scaling argument then yields
\[
h^{-1}\|w-w_c\|_{L^2(\Gamma_h)}+ |w_c|_{H^1(\Gamma_h)} \lesssim |w|_{H^1_h(\Gamma_h)}.
\]
Define its zero-mean counterpart $\mathring{w}_c := w_c - \frac{1}{|\Gamma_h|} \int_{\Gamma_h} w_c \mathrm{d}\sigma_h \in \mathring{H}^1(\Gamma_h)$. The uniform Poincar\'e inequality on the closed surface $\Gamma_h$ (see \cite[Section 4.2.1]{bonito2020finite}) yields
\[
\|\mathring{w}_{c}\|_{L^2(\Gamma_h)} \lesssim |\mathring{w}_{c} |_{H^1(\Gamma_h)} = |w_c|_{H^1(\Gamma_h)},
\]
where the hidden constant is independent of $\Gamma_h$.

Using these inequalities and the zero-mean property of $w$, we obtain
\begin{align*}
\|w\|_{L^2(\Gamma_h)} & \leq \|w-w_c\|_{L^2(\Gamma_h)} + \|\mathring{w}_{c}\|_{L^2(\Gamma_h)} 
+ \frac{1}{|\Gamma_h|^{1/2}} \big| \int_{\Gamma_h} (w_c-w) \mathrm{d} \sigma_h\big| \\
&\lesssim \|w-w_c\|_{L^2(\Gamma_h)} + |w_{c}|_{H^1(\Gamma_h)} 
\lesssim  |w|_{H_h^1(\Gamma_h)},
\end{align*}
which establishes the desired inequality.
\end{proof}

\begin{lemma}[discrete Korn's inequality for trace-free Hessian]\label{lem:discrete_korn}
It holds that
\begin{align}\label{eq:discrete_korn_stmt}
\| \nabla_{\Gamma_h}^2 w \|_{L^2(\Gamma_h)}^2
\lesssim
\| \HG w \|_{L^2(\Gamma_h)}^2
+ \| \nabla_{\Gamma_h} w \|_{L^2(\Gamma_h)}^2
+ h^{-3}\| [w]\|_{L^2(\mathcal{E}_h)}^2 \qquad \forall w \in V_h.
\end{align}
Consequently, by discrete Poincar\'e inequality \eqref{eq:poincare},
\begin{align}\label{eq:Hessian-H-eq}
\| w\|_{H^2_h(\Gamma_h)} \eqsim \triplenorm{w}\qquad \forall w \in \mathring{V}_h.
\end{align}
\end{lemma}

\begin{proof}
The proof combines the existing $H^1$-conforming relative with a
vector-Korn estimate.

\emph{Step 1 (reduction using the $H^1$-conforming relative).}
Let $w_c:=\Pi_h^c w$ and set $\bv_c:=\curl_{\Gamma_h}w_c$.
Summing \eqref{eq:sNZT-relative-1} with $m=1,2$ over the mesh and using
quasi-uniformity gives
\begin{equation}\label{eq:korn-relative-defect}
\|\nabla_{\Gamma_h}^2(w_c-w)\|_{L^2(\Gamma_h)}^2
+h^{-2}\|\nabla_{\Gamma_h}(w_c-w)\|_{L^2(\Gamma_h)}^2
\lesssim
h^{-3}\|[w]\|_{L^2(\mathcal E_h)}^2.
\end{equation}
Thus, by the triangle inequality, it remains only to prove
\begin{equation}\label{eq:korn-relative-target}
\|\nabla_{\Gamma_h}^2w_c\|_{L^2(\Gamma_h)}^2
\lesssim
\|\HG w\|_{L^2(\Gamma_h)}^2
+\|\nabla_{\Gamma_h}w\|_{L^2(\Gamma_h)}^2
+h^{-3}\|[w]\|_{L^2(\mathcal E_h)}^2.
\end{equation}
Because $\bnu_h$ is constant on each element,
$\nabla_{\Gamma_h}\bv_c=\bnu_h^\times\nabla_{\Gamma_h}^2w_c$.
Consequently, $\|\nabla_{\Gamma_h}\bv_c\|
=\|\nabla_{\Gamma_h}^2w_c\|$,
$\mathrm E_{\Gamma_h}\bv_c=\HG w_c$, and
$\|\bv_c\|=\|\nabla_{\Gamma_h}w_c\|$. We establish
\eqref{eq:korn-relative-target} by applying the discrete Korn inequality
to $\bv_c$ in Step~2 and controlling its ordinary vector jumps in Step~3.

\emph{Step 2 (discrete Korn for $H(\div;\Gamma_h)$).}
Since $w_c$ is continuous and piecewise quadratic, its tangential
derivative is single-valued on every edge. Thus
$[\bv_c\cdot\bn_h]=-[\nabla_{\Gamma_h}w_c\cdot\bt_h]=0$, and
$\bv_c$ belongs to the $H(\div_{\Gamma_h})$-conforming surface
$\mathrm{BDM}_1$ space by \cite[Lemma~B.1]{neilan2024c}. The
discrete Korn inequality \cite[(B.4)]{neilan2024c}, with the
ordinary vector jump defined by \eqref{eq:ordinary-jump-average},
then gives
\[
\|\nabla_{\Gamma_h}^2w_c\|_{L^2(\Gamma_h)}^2
\lesssim
\|\HG w_c\|_{L^2(\Gamma_h)}^2
+\|\nabla_{\Gamma_h}w_c\|_{L^2(\Gamma_h)}^2
+\sum_{e\in\mathcal E_h}h_e^{-1}
\|\llbracket\bv_c\rrbracket\|_{L^2(e)}^2.
\]

\emph{Step 3 (control of the ordinary vector jumps).}
Fix $e=\partial K_1\cap\partial K_2$ and orient the edge so that
$\bm t_2=-\bm t_1$. On $K_i$, we have
$\bv_c=(\nabla_{\Gamma_h}w_c\cdot\bm n_i)\bm t_i
-(\nabla_{\Gamma_h}w_c\cdot\bm t_i)\bm n_i$. The continuity of the
tangential derivative of $w_c$ gives the exact decomposition
\begin{equation}\label{eq:korn-relative-vector-jump}
\llbracket\bv_c\rrbracket
=
[\nabla_{\Gamma_h}w_c\cdot\bn_h] \,\bm t_1
-
(\nabla_{\Gamma_h}w_c|_{K_1}\cdot\bm t_1)
(\bm n_1+\bm n_2).
\end{equation}
Since $\bm n_i=\bm t_i\times\bm\nu_i$ and $\bm t_2=-\bm t_1$,
the geometric approximation gives $|\bm n_1+\bm n_2|\lesssim h_e$.
An inverse trace estimate then yields
\begin{equation}\label{eq:korn-relative-vector-jump-bound}
h_e^{-1}\|\llbracket\bv_c\rrbracket\|_{L^2(e)}^2
\lesssim
h_e^{-1}\|[\nabla_{\Gamma_h}w_c\cdot\bn_h]\|_{L^2(e)}^2
+
\|\nabla_{\Gamma_h}w_c\|_{L^2(\omega_e)}^2.
\end{equation}

The function $[\nabla_{\Gamma_h}w\cdot\bn_h]\in\mathcal P_1(e)$
has zero mean by the Morley continuity condition, while
\eqref{eq:H-Hessian} gives
$\partial_{\bm t_1}[\nabla_{\Gamma_h}w\cdot\bn_h]
=[\bm t_h^\top\HG w\,\bm t_h]$. Hence the one-dimensional
Poincar\'e inequality and an inverse trace estimate give
$h_e^{-1}\|[\nabla_{\Gamma_h}w\cdot\bn_h]\|_{L^2(e)}^2
\lesssim\|\HG w\|_{L^2(\omega_e)}^2$. Since
$w_c=w+(w_c-w)$, another inverse trace estimate gives
\begin{equation}\label{eq:korn-relative-normal-jump}
h_e^{-1}\|[\nabla_{\Gamma_h}w_c\cdot\bn_h]\|_{L^2(e)}^2
\lesssim
\|\HG w\|_{L^2(\omega_e)}^2
+
h_e^{-2}\|\nabla_{\Gamma_h}(w_c-w)\|_{L^2(\omega_e)}^2.
\end{equation}
Combining \eqref{eq:korn-relative-vector-jump-bound},
\eqref{eq:korn-relative-normal-jump}, and \eqref{eq:korn-relative-defect},
and summing over the edges, we obtain
\begin{equation}\label{eq:korn-relative-jump-sum}
\begin{aligned}
\sum_{e\in\mathcal E_h}h_e^{-1}
\|\llbracket\bv_c\rrbracket\|_{L^2(e)}^2
\lesssim{}&
\|\HG w\|_{L^2(\Gamma_h)}^2
+\|\nabla_{\Gamma_h}w\|_{L^2(\Gamma_h)}^2
+h^{-3}\|[w]\|_{L^2(\mathcal E_h)}^2.
\end{aligned}
\end{equation}

Finally, $\|\HG w_c\|\lesssim
\|\HG w\|+\|\nabla_{\Gamma_h}^2(w_c-w)\|$ and
$\|\nabla_{\Gamma_h}w_c\|\lesssim
\|\nabla_{\Gamma_h}w\|+\|\nabla_{\Gamma_h}(w_c-w)\|$.
Substituting \eqref{eq:korn-relative-defect} and
\eqref{eq:korn-relative-jump-sum} into the vector Korn estimate proves
\eqref{eq:korn-relative-target}. Combining this estimate with
\eqref{eq:korn-relative-defect} and
$\nabla_{\Gamma_h}^2w=\nabla_{\Gamma_h}^2w_c
-\nabla_{\Gamma_h}^2(w_c-w)$ yields \eqref{eq:discrete_korn_stmt}.
\end{proof}

Finally, since the energy norm contains the jump terms, for the 
$H^1$-conforming interpolant defined in \eqref{eq:Pi_h^c}, the estimate 
\eqref{eq:sNZT-relative-1} implies
\begin{equation}\label{eq:sNZT-relative}
\| w - \Pi_h^c w \|_{H^m(\Gamma_h)}
\lesssim
h^{2-m} \triplenorm{w}
\quad \forall w \in V_h,\ 0 \le m \le 2.
\end{equation}

\section{Error estimates} \label{sec:analysis}
This section establishes the error estimates for the scheme~\eqref{eq:FEM}.
Throughout this section, we assume that the exact surface $\gamma$ is of
class $C^4$, the standing regularity requirement under which all the estimates
below are derived. In particular, it supplies, through the elliptic regularity
\eqref{eq:regularity}, the $H^3$ regularity of the stream function that the
analysis relies on.
We follow a Strang-type argument: the consistency errors of the bilinear
form $a_h$ and of the load functional $l_h$ are estimated in
Sections~\ref{subsec:consistency-a} and~\ref{subsec:consistency-l}, and
then combined to prove the energy error estimate
(Section~\ref{subsec:energy-estimate}) and the optimal broken-$H^1$
error estimate, together with the resulting velocity $L^2$ estimate
(Section~\ref{subsec:H1-estimate}).

\subsection{Consistency of the bilinear form}\label{subsec:consistency-a}
We begin by collecting the geometric consistency estimates for the
bilinear form $a_h$.
By combining the geometric approximation properties of the surface trace-free Hessian operator (see, e.g., \cite[Lemma~5.3]{neilan2024c}) with the interpolation error estimate \eqref{eq:I-L0}, we have the following result.
\begin{lemma}[trace-free Hessian approximation] \label{lm:Hessian-interpolation}
For $w \in H^3(\gamma)$, it holds that 
\begin{align}
\|( \mathrm{H}_\gamma w)^e - \mathrm{H}_{\Gamma_h} w^e \|_{L^2(\Gamma_h)} &\lesssim h \|w\|_{H^3(\gamma)}, \label{eq:Hessian-1}\\
\|( \mathrm{H}_\gamma w)^e - \mathrm{H}_{\Gamma_h} I_h w\|_{L^2(\Gamma_h)} &\lesssim h \|w\|_{H^3(\gamma)}. \label{eq:Hessian-interpolation}
\end{align}
\end{lemma}

The next estimate follows from pointwise geometric bounds comparing the surface curl and divergence operators on the exact surface $\gamma$ and the discrete surface $\Gamma_h$. It is stated for an $H^1$-conforming function, which is the setting in which it is applied (to the conforming relatives $\psi_c$, $\epsilon_c$ below).
\begin{lemma}[curl--divergence approximation] \label{lm:geom-biharmonic}
Let $w \in H^3(\gamma)$, let $v_h \in V_h$, and let $v_c := \Pi_h^c v_h \in C^0(\Gamma_h)$ be its $H^1$-conforming relative defined in \eqref{eq:Pi_h^c}, with lift $v_c^\ell \in H^1(\gamma)$. Then,
\begin{equation} \label{eq:geo-1}
\begin{aligned}
\left| -(\div_\gamma (\Hg w), \curl_\gamma v_c^\ell)_\gamma
+(\div_{\Gamma_h} (\Hg w)^e, \curl_{\Gamma_h} v_c)_{\Gamma_h} \right| &\lesssim h \|w\|_{H^3(\gamma)} \|v_c\|_{H^1(\Gamma_h)} \\
&\lesssim h \|w\|_{H^3(\gamma)} \triplenorm{v_h}.
\end{aligned}
\end{equation}
\end{lemma}
\begin{proof}[Sketch of proof]
Set $\bA := \Hg w$. Transforming the integral over $\gamma$ to $\Gamma_h$ through the lift and using $\mathrm{d}\sigma = \mu_h \,\mathrm{d}\sigma_h$ with $|1-\mu_h| \lesssim h^2$, it suffices to compare the two integrands pointwise on $\Gamma_h$. The surface divergence and curl are built from the tangential projection and the normal, which satisfy $|\bP - \bP_h| \lesssim h$ and $|\bnu - \bnu_h| \lesssim h$. Hence the standard pointwise geometric bounds (cf.~\cite{dziuk2013finite,demlow2009higher}) give
\[
\big|\div_{\Gamma_h} \bA^e - (\div_\gamma \bA)^e\big| \lesssim h \big(|\bA| + |\nabla_\gamma \bA|\big)^e,
\qquad
\big|\curl_{\Gamma_h} v_c - (\curl_\gamma v_c^\ell)^e\big| \lesssim h |\nabla_{\Gamma_h} v_c|.
\]
Inserting these into the pairing, together with the area estimate and the Cauchy--Schwarz inequality yield the first bound. The second follows from the relative estimate \eqref{eq:sNZT-relative}, which gives $\|v_c\|_{H^1(\Gamma_h)} \lesssim \triplenorm{v_h}$.
\end{proof}

\paragraph{Second-order consistency of the trace-free Hessian form.}
The form $j_h$ defined below collects the consistency error of $a_h$ associated with the trace-free Hessian. The first-order geometric consistencies above are routine; its second-order consistency is the main difficulty, and it is here that the geometric estimate of Section~\ref{sec:geo-error} is essential.

Define $j_h : H^3(\gamma) \times H_h^2(\Gamma_h) \to \mathbb{R}$ as follows:
\begin{equation}\label{eq:jump-part}
j_h(w,\psi_h) := ((\Hg w)^e, \HG \psi_h)_{\Gamma_h} + (\div_{\Gamma_h} (\Hg w)^e, \curl_{\Gamma_h} \psi_h)_{\Gamma_h}.
\end{equation}

\begin{lemma}[consistency of $j_h$]\label{lm:jump-estimate}
For $w \in H^3(\gamma)$, it holds that 
\begin{align} 
|j_h(w,\psi_h)| &\lesssim h \|w\|_{H^3(\gamma)} \triplenorm{\psi_h}  \quad \forall \psi_h \in V_h, \label{eq:jump-estimate1}\\
|j_h(w,I_h \psi)| &\lesssim h^2 \|w\|_{H^3(\gamma)} \|\psi\|_{H^3(\gamma)}  \quad \forall \psi \in H^3(\gamma). \label{eq:jump-estimate-Ih} \\
|j_h(w,\psi^e)| &\lesssim h^2 \|w\|_{H^3(\gamma)} \|\psi\|_{H^3(\gamma)}\quad \forall \psi\in H^3(\gamma). \label{eq:jump-estimate-e} 
\end{align}
\end{lemma}
\begin{proof}
The second-order estimates draw on the geometric consistency results of
Section~\ref{sec:geo-error} as follows: \eqref{eq:jump-estimate-e} is a
direct specialization of Corollary~\ref{cor:stokes-tensor-green}
(Stokes-type tensor Green consistency), while the second-order part of
\eqref{eq:jump-estimate-Ih} relies on the conormal-flux consistency
\eqref{eq:geo-trace-2}.
First, the pure geometric estimate \eqref{eq:jump-estimate-e} follows from \eqref{eq:jh-extension} by taking $\bA = \Hg w$. For the other two estimates, decompose $j_h$ as follows:
\begin{equation*}
j_h(w,\psi_h) =\sum_{T\in \mathcal{T}_h} \int_{\partial T} ( (\Hg w)^e\bn_h)\cdot \curl_{\Gamma_h} \psi_h \, \mathrm{d}s_h
= \sum_{i=1}^4 j_i(w,\psi_h).
\end{equation*}
where 
\begin{equation*}
\begin{aligned}
j_1(w,\psi_h) &:= -\langle \{\bm{n}_h^\top (\Hg w)^e \, \bm{n}_h \} , [\nabla_{\Gamma_h} \psi_h\cdot\bm{t}_h]\rangle_{\mathcal{E}_h}, \\
j_2(w,\psi_h) &:= -\langle [\bm{n}_h^\top (\Hg w)^e \, \bm{n}_h ] , \{ \nabla_{\Gamma_h} \psi_h\cdot\bm{t}_h\} \rangle_{\mathcal{E}_h}, \\
j_3(w,\psi_h) &:= \langle \{\bm{t}_h^\top (\Hg w)^e \, \bm{n}_h \} , [\nabla_{\Gamma_h} \psi_h\cdot \bm{n}_h]\rangle_{\mathcal{E}_h}, \\
j_4(w,\psi_h) &:= \langle [\bm{t}_h^\top (\Hg w)^e \, \bm{n}_h ] , \{ \nabla_{\Gamma_h} \psi_h\cdot\bm{n}_h \} \rangle_{\mathcal{E}_h}.
\end{aligned}
\end{equation*}

We estimate the four components of $j_h(w,\psi_h)$ separately.
For $\psi_h\in V_h$, using the weak continuity \eqref{eq:weak-continuity}, the trace inequality, and standard approximation theory, we obtain
\begin{equation}\label{eq:j1-h}
\begin{aligned}
|j_1(w,\psi_h)| &\leq  \|\{\bm{n}_h^\top (\Hg w)^e \, \bm{n}_h \} - C_e \|_{L^2(\mathcal{E}_h)} \|[\nabla_{\Gamma_h} \psi_h \cdot \bm{t}_h]\|_{L^2(\mathcal{E}_h)} \\
&\lesssim h \|w\|_{H^3(\gamma)} \triplenorm{\psi_h},
\end{aligned}
\end{equation}
where $C_e=P_e^0 \{\bm{n}_h^\top (\Hg w)^e \, \bm{n}_h \}$ and $P_e^0: L^2(\mathcal{E}_h) \to \mathcal{P}_0(\mathcal{E}_h)$ is the local $L^2$ projection onto the piecewise constants on the skeleton. An analogous argument gives $|j_3(w,\psi_h)|\lesssim h \|w\|_{H^3(\gamma)} \triplenorm{\psi_h}$.

By the geometric bound $\big|\bP \llbracket \bn_h \rrbracket \big|\lesssim h^2$ in Lemma \ref{lm:Pnjump}, we deduce
\[
\big|[\bm{t}_h^\top (\Hg w)^e \, \bm{n}_h ]\big|\lesssim \big| (\Hg w)^e \big|\, \big| \bP \llbracket \bn_h \rrbracket \big| \lesssim h^2|(\Hg w)^e|,
\]
and further,
\begin{equation}
\begin{aligned}
\big|[\bm{n}_h^\top (\Hg w)^e \, \bm{n}_h ]\big| &= \big| (\bm{n}_1+\bm{n}_2)^\top (\Hg w)^e \, (\bm{n}_1-\bm n_2)  \big| \\
&= \big| (\bm{n}_1-\bm{n}_2)^\top (\Hg w)^e \, \bP \llbracket \bn_h \rrbracket \big| \lesssim h^2|(\Hg w)^e|,
\end{aligned}
\end{equation}
where both bounds use the tangentiality of $(\Hg w)^e$ to replace $\bm{n}_1 + \bm{n}_2$ by $\bP \llbracket \bn_h \rrbracket$ in the last factor. These bounds immediately imply
\begin{equation*}
\begin{aligned}
&\quad |j_2(w,\psi_h)|+|j_4(w,\psi_h)| \\
&\lesssim h^2 \|(\Hg w)^e \|_{L^2(\mathcal{E}_h)} \Big( \|\{\nabla_{\Gamma_h} \psi_h \cdot \bm{t}_h \}\|_{L^2(\mathcal{E}_h)}+ \|\{\nabla_{\Gamma_h} \psi_h \cdot \bm{n}_h \}\|_{L^2(\mathcal{E}_h)}\Big) \\
&\lesssim h \|w\|_{H^3(\gamma)} \triplenorm{\psi_h}.
\end{aligned}
\end{equation*}
Combining the estimates for all four parts completes the proof of \eqref{eq:jump-estimate1}.

A weak continuity argument similar to \eqref{eq:j1-h}, combined with the jump estimates for the interpolation \eqref{eq:normal-jump}--\eqref{eq:t-jump}, yields
\begin{equation*}
|j_1(w, I_h \psi)| + |j_3(w, I_h \psi)| \lesssim h^2 \|w\|_{H^3(\gamma)} \|\psi \|_{H^3(\gamma)}.
\end{equation*}
Recombining the two components in the edge frame, $j_2+j_4$ is the pairing of the vector jump $\llbracket (\Hg w)^e\bm n_h\rrbracket$ with the average discrete curl,
\begin{equation*}
j_2(w, I_h \psi) + j_4(w, I_h \psi)
= \big\langle \llbracket (\Hg w)^e\bm n_h\rrbracket,\ \{\curl_{\Gamma_h}I_h \psi\}\big\rangle_{\mathcal{E}_h}.
\end{equation*}
Then, combining \eqref{eq:geo-trace-2} with the approximation properties \eqref{eq:I-L0} and \eqref{eq:Hessian-H-eq}, we obtain
\begin{equation*}
\begin{aligned}
&~\quad |j_2(w, I_h \psi) + j_4(w, I_h \psi)|\\
&\leq |\langle\llbracket (\Hg w)^e\bm n_h\rrbracket,\ (\curl_\gamma\psi)^e\rangle_{\mathcal{E}_h}|+|\langle\llbracket (\Hg w)^e\bm n_h\rrbracket,\ \{\curl_{\Gamma_h}I_h \psi \}-(\curl_\gamma\psi)^e\rangle_{\mathcal{E}_h}|\\
& \lesssim h^2 \|w\|_{H^3(\gamma)} \|\psi\|_{H^2(\gamma)}+ h^{3/2}\|\Hg w \|_{H^1(\gamma)} h^{1/2}\|\psi\|_{H^2(\gamma)} \\
&\lesssim h^2 \|w\|_{H^3(\gamma)} \|\psi\|_{H^2(\gamma)}.
\end{aligned}
\end{equation*}
This proves \eqref{eq:jump-estimate-Ih}.
\end{proof}

The following lemma is the key second-order estimate of the analysis. Whereas the individual consistency terms above are only first order, it shows that the consistency error of the form $a_h$ (the part of the discrete bilinear form without the stabilization $s_h^0$, whose consistency is treated separately in Section~\ref{subsec:H1-estimate}) is second order. This is precisely the estimate that lifts the broken-$H^1$ error to the optimal second order in the duality argument of Section~\ref{subsec:H1-estimate}. It is here that the second-order geometric consistency of the trace-free Hessian form, established through $j_h$, is used in full.
\begin{lemma}[second-order consistency of $a_h$] \label{lm:equivalence-consistency}
For any $w, \psi \in H^3(\gamma)$, it holds that 
\begin{equation} \label{eq:bilinear-approx-equiv}
    \big| a_h(I_h w, I_h \psi) - a(w, \psi) \big| \lesssim h^2 \|w\|_{H^3(\gamma)} \|\psi\|_{H^3(\gamma)}.
\end{equation}
\end{lemma}
\begin{proof}
Denote $w_I := I_h w$ and $\psi_I := I_h \psi$. The following decomposition holds:
$$
\begin{aligned}
 a_h(w_I, \psi_I) - a(w, \psi)
&=  (\HG w_I-(\Hg w)^e, (\Hg \psi)^e)_{\Gamma_h} + ((\Hg w)^e,\HG \psi_I -(\Hg \psi)^e)_{\Gamma_h} \\
&\quad + (\HG w_I-(\Hg w)^e, \HG \psi_I -(\Hg \psi)^e)_{\Gamma_h}\\
&\quad + ((\Hg w)^e, (\Hg \psi)^e)_{\Gamma_h}-(\Hg w, \Hg \psi)_{\gamma}\\
&\quad +  (\nabla_{\Gamma_h} w_I, \nabla_{\Gamma_h} \psi_I)_{\Gamma_h} - (\nabla_\gamma w, \nabla_\gamma \psi)_\gamma \\
&:= I_1+I_2+I_3+I_4 .
\end{aligned}
$$
This can be further decomposed as follows:
\begin{equation*}
\begin{aligned}
I_1&:= (\HG w_I-(\Hg w)^e, (\Hg \psi)^e)_{\Gamma_h} + ((\Hg w)^e,\HG \psi_I -(\Hg \psi)^e)_{\Gamma_h}\\
&= (\HG w^e-(\Hg w)^e, (\Hg \psi)^e)_{\Gamma_h} + ((\Hg w)^e,\HG \psi^e -(\Hg \psi)^e)_{\Gamma_h}\\
&\quad -(\div_{\Gamma_h}(\Hg \psi)^e, \curl_{\Gamma_h} (w_I - w^e))_{\Gamma_h} -(\div_{\Gamma_h}(\Hg w)^e, \curl_{\Gamma_h} (\psi_I - \psi^e))_{\Gamma_h} \\
&\quad +j_h(\psi,w_I - w^e)+j_h(w,\psi_I-\psi^e)\\
&:= J_1+J_2+J_3
\end{aligned}
\end{equation*}

The remaining terms are controlled as follows:
\begin{itemize}
\item $I_2$: by the trace-free Hessian consistency \eqref{eq:Hessian-interpolation};
\item $I_3$: by the geometric measure approximation $|1-\mu_h|\lesssim h^2$;
\item $I_4$: by the superconvergence estimate \eqref{eq:first-order-tangential};
\item $J_1,J_2$: by the trace-free Hessian consistency \eqref{eq:A-Hg-HG-estimate} and the interpolation estimate \eqref{eq:I-L0};
\item $J_3$: by the consistency estimates \eqref{eq:jump-estimate-Ih} and \eqref{eq:jump-estimate-e}.
\end{itemize}
Each term is bounded by $h^2 \|w\|_{H^3(\gamma)}\|\psi\|_{H^3(\gamma)}$, which completes the proof.
\end{proof}

\subsection{Consistency of the load functional}\label{subsec:consistency-l}
For the source term, we assume that
$\|\bm{f} - \bm{f}_h^\ell\|_{L^2(\gamma)} \lesssim h^2 \|\bm{f}\|_{L^2(\gamma)}$,
which is satisfied, for instance, by choosing $\bm{f}_h = \bm{f}^e$.
Then, by \cite[Lemma~5.1]{neilan2024c}, together with the $H^1$-conforming approximation estimate \eqref{eq:sNZT-relative} and the interpolation estimate \eqref{eq:I-L0}, we obtain the following result.

\begin{lemma}[source consistency]
The following estimates hold:
\begin{subequations}
\begin{align}
| l((\Pi_h^c \psi_h)^\ell) - l_h(\psi_h)|
&\lesssim h \|\bm{f}\|_{L^2(\gamma)} \triplenorm{\psi_h}
\quad \forall \psi_h \in \mathring{V}_{h},
\label{eq:source-estimate1} \\
| l(\psi)-l_h(I_h\psi)|
&\lesssim h^2 \|\bm{f}\|_{L^2(\gamma)} \|\psi \|_{H^3(\gamma)}
\quad \forall \psi \in H^3(\gamma).
\label{eq:source-estimate2}
\end{align}
\end{subequations}
\end{lemma}

\subsection{Energy norm error estimate}\label{subsec:energy-estimate}
With the consistency estimates of the bilinear form and the load functional in hand, we now combine them through a Strang-type argument to obtain the energy norm error estimate. Here, only the first-order consistency is needed.

\begin{theorem}[energy norm error estimate] \label{tm:energy-estimate}
For $C^4$ surface $\gamma$, let $\phi$ be the exact solution of the surface Stokes problem in stream-function form \eqref{eq:s-streamStokes}, and let $\phi_h \in \mathring{V}_{h}$ be the solution of \eqref{eq:FEM}. Then, 
\begin{equation} \label{eq:energy-estimate}
\|\phi^e - \phi_h\|_{H^2_h(\Gamma_h)} \lesssim h(\|\phi\|_{H^3(\gamma)}+ \|\bm{f}\|_{L^2(\gamma)}) \lesssim h \|\bm{f}\|_{L^2(\gamma)}. 
\end{equation}
\end{theorem}
\begin{proof}
Define $\phi_I := I_h \phi \in \mathring{V}_h$ and let $\psi_c := \Pi_h^c \psi_h\in C^0(\Gamma_h)$ be the $H^1$-conforming relative defined in \eqref{eq:Pi_h^c}. Utilizing coercivity, it follows that
\begin{equation}\label{eq:energy-dep}
\begin{aligned}
\triplenorm{\phi_I - \phi_h} &\lesssim \sup_{\psi_h \in \mathring{V}_h} \frac{a_h(\phi_I - \phi_h, \psi_h)+s_h^0(\phi_I - \phi_h, \psi_h) }{\triplenorm{\psi_h}}\\
&\leq \sup_{\psi_h \in \mathring{V}_h}\frac{|a_h(\phi_I, \psi_h)-l(\psi_c^\ell)|}{\triplenorm{\psi_h}} 
+\sup_{\psi_h \in \mathring{V}_h} \frac{|s_h^0(\phi_I,\psi_h)|}{\triplenorm{\psi_h}}
+ \sup_{\psi_h \in \mathring{V}_h} \frac{|l(\psi_c^\ell)- l_h(\psi_h)|}{\triplenorm{\psi_h}}.
\end{aligned}    
\end{equation}
Since $\psi_c \in H^1(\Gamma_h)$, the term $l(\psi_c^\ell)$ can be expanded using \eqref{eq:l_continue}, which allows for the decomposition of $a_h(\phi_I, \psi_h) - l(\psi_c^\ell)$ as follows:
\begin{equation*}
\begin{aligned}
a_h(\phi_I, \psi_h) - l(\psi_c^\ell)&= (\HG \phi_I, \HG \psi_h)_{\Gamma_h}  + (\div_\gamma (\Hg \phi), \curl_\gamma \psi_c^\ell)_\gamma \\
&\quad +(\nabla_{\Gamma_h} \phi_I,\nabla_{\Gamma_h} \psi_h)_{\Gamma_h}-(\nabla_\gamma \phi,\nabla_\gamma  \psi_c^\ell )_\gamma = \sum_{i=1}^5 I_{i},
\end{aligned}
\end{equation*}
where
\begin{equation*} 
\begin{aligned}
I_{1}&:= (\HG \phi_I - (\Hg \phi)^e, \HG \psi_h)_{\Gamma_h} \\
I_{2}&:=  ((\Hg \phi)^e, \HG \psi_h)_{\Gamma_h} + (\div_{\Gamma_h} (\Hg \phi)^e, \curl_{\Gamma_h} \psi_h)_{\Gamma_h} \\
I_{3}&:=  - (\div_{\Gamma_h} (\Hg \phi)^e, \curl_{\Gamma_h}(\psi_h - \psi_c))_{\Gamma_h} \\
I_{4}&:=  - (\div_{\Gamma_h} (\Hg \phi)^e, \curl_{\Gamma_h}\psi_c)_{\Gamma_h} + (\div_\gamma (\Hg \phi), \curl_\gamma \psi_c^\ell)_\gamma \\
I_{5}&:= (\nabla_{\Gamma_h} \phi_I,\nabla_{\Gamma_h} \psi_h)_{\Gamma_h}-(\nabla_\gamma \phi,\nabla_\gamma  \psi_c^\ell )_\gamma.
\end{aligned}
\end{equation*}

The first four terms are estimated as follows, each bounded by $h \|\phi\|_{H^3(\gamma)}\triplenorm{\psi_h}$:
\begin{itemize}
\item $I_1,I_2$: by the trace-free Hessian consistency \eqref{eq:Hessian-interpolation} and the jump estimate \eqref{eq:jump-estimate1};
\item $I_3$: by the relative estimate \eqref{eq:sNZT-relative}, since $|I_3|\lesssim \|\phi\|_{H^3(\gamma)}\,|\psi_h-\psi_c|_{H^1_h(\Gamma_h)}$;
\item $I_4$: by the curl--divergence approximation Lemma~\ref{lm:geom-biharmonic} applied to $\psi_c=\Pi_h^c\psi_h$.
\end{itemize}
For $I_5$, we write
\[
\begin{aligned}
I_5 = & (\nabla_{\Gamma_h}\phi_I,\nabla_{\Gamma_h}(\psi_h-\psi_c))_{\Gamma_h}
+ (\nabla_{\Gamma_h}(\phi_I-\phi^e),\nabla_{\Gamma_h}\psi_c)_{\Gamma_h} \\
& + \big[(\nabla_{\Gamma_h}\phi^e,\nabla_{\Gamma_h}\psi_c)_{\Gamma_h} - (\nabla_\gamma\phi,\nabla_\gamma\psi_c^\ell)_\gamma\big],
\end{aligned}
\]
the first two terms are bounded by $h\|\phi\|_{H^3(\gamma)}\triplenorm{\psi_h}$ using the relative estimate \eqref{eq:sNZT-relative} and the interpolation estimate \eqref{eq:I-L0}, while the last term is $O(h^2)$ by the superconvergence estimate \eqref{eq:first-order-tangential}. Hence $|I_5| \lesssim h \|\phi\|_{H^3(\gamma)}\triplenorm{\psi_h}$.
Combining the estimates for $I_1,\dots,I_5$, we have
\begin{equation} \label{eq:ah-consistency}
|a_h(\phi_I, \psi_h) - l( \psi_c^\ell)| \lesssim h \|\phi\|_{H^3(\gamma)} \triplenorm{\psi_h}.
\end{equation}

By the jump estimate of the interpolation \eqref{eq:value-jump}, the stabilization term satisfies
\begin{equation} \label{eq:sh-bound}
|s_h^0(\phi_I,\psi_h)| \lesssim h \|\phi\|_{H^3(\gamma)} \triplenorm{\psi_h}.
\end{equation}

Substituting \eqref{eq:ah-consistency}, \eqref{eq:sh-bound}, and the source term estimate \eqref{eq:source-estimate1} into the decomposition \eqref{eq:energy-dep}, yields
$\triplenorm{\phi_I - \phi_h}\lesssim h(\|\phi\|_{H^3(\gamma)} + \|\bm{f} \|_{L^2(\gamma)})$.
Applying the triangle inequality and the norm equivalence \eqref{eq:Hessian-H-eq} completes the proof.
\end{proof}

\subsection{Broken $H^1$ error estimate and velocity $L^2$ estimate} \label{subsec:H1-estimate}
We now establish the second-order estimate in the broken $H^1$ norm by an Aubin--Nitsche duality argument. The point is that this optimal order is attained under only the natural regularity: a $C^4$ surface and a stream function $\phi \in H^3(\gamma)$, which is exactly the elliptic regularity afforded by \eqref{eq:s-streamStokes} for an $L^2$ right-hand side, with no additional smoothness assumed. What makes this possible is the sharp geometric consistency of Section~\ref{sec:geo-error}, entering through Lemma~\ref{lm:equivalence-consistency} (the second-order consistency of $a_h$). The resulting broken $H^1$ error estimate in turn yields the optimal $L^2$ estimate for the recovered velocity.
\begin{theorem}[broken $H^1$ error estimate] \label{tm:H1-estimate}
For $C^4$ surface $\gamma$, let $\phi$ be the exact solution of the surface Stokes problem in stream-function form \eqref{eq:s-streamStokes}, and let $\phi_h \in \mathring{V}_{h}$ be the discrete solution of \eqref{eq:FEM}. It holds that
\begin{equation}\label{eq:H1-estimate}
    \| \phi^e - \phi_h \|_{H^1_h(\Gamma_h)} \lesssim h^2  (\|\phi\|_{H^3(\gamma)} + \| \bm{f} \|_{L^2(\gamma)}) \lesssim h^2 \|\bm{f}\|_{L^2(\gamma)}.
\end{equation}
\end{theorem}
\begin{proof}
We denote $\phi_I := I_h \phi$ and $\epsilon_h := \phi_I - \phi_h \in \mathring{V}_h$, and let $\epsilon_c := \Pi_h^c \epsilon_h$ be the $H^1$-conforming relative defined in \eqref{eq:Pi_h^c}, which satisfies $\epsilon_c^\ell \in H^1(\gamma)$. Setting $g := - \Delta_\gamma \epsilon_c^\ell \in \mathring{H}^{-1}(\gamma)$, we consider the auxiliary problem
\begin{equation} \label{eq:aux-problem-sub}
\frac{1}{2}\Delta_\gamma^2 w +\div_\gamma((K-1)\nabla_\gamma w) = g \quad \text{on } \gamma,
\end{equation}
whose solution satisfies $\|w\|_{H^3(\gamma)} \lesssim \|g\|_{\mathring{H}^{-1}(\gamma)}$ by the elliptic regularity \eqref{eq:regularity}. Since $\|g\|_{\mathring{H}^{-1}(\gamma)} = \|\Delta_\gamma \epsilon_c^\ell \|_{\mathring{H}^{-1}(\gamma)} \lesssim \|\epsilon_c^\ell\|_{H^1(\gamma)}$, we obtain
\begin{equation} \label{eq:regularity-bound}
\|w\|_{H^3(\gamma)} \lesssim \|g\|_{\mathring{H}^{-1}(\gamma)} \lesssim \|\epsilon_c^\ell\|_{H^1(\gamma)} \lesssim |\epsilon_c^\ell|_{H^1(\gamma)}.
\end{equation}

Next, testing the auxiliary equation \eqref{eq:aux-problem-sub} with $\epsilon_c^\ell$ gives
\begin{equation*}
\begin{aligned}
|\epsilon_c^\ell|_{H^1(\gamma)}^2 &= (g,\epsilon_c^\ell)_\gamma
= -(\div_\gamma (\Hg w), \curl_\gamma \epsilon_c^\ell)_\gamma +(\nabla_\gamma w,\nabla_\gamma  \epsilon_c^\ell )_\gamma \\
& = \sum_{i=1}^5 I_{i}+a_h(w_I, \epsilon_h),
\end{aligned}
\end{equation*}
where 
\begin{equation*}
\begin{aligned}
I_{1}& := -(\div_\gamma (\Hg w), \curl_\gamma \epsilon_c^\ell)_\gamma +(\div_{\Gamma_h} (\Hg w)^e, \curl_{\Gamma_h} \epsilon_c)_{\Gamma_h}, \\
I_{2}& :=  (\nabla_\gamma w,\nabla_\gamma  \epsilon_c^\ell )_\gamma-(\nabla_{\Gamma_h} w^e,\nabla_{\Gamma_h} \epsilon_c)_{\Gamma_h}, \\
I_{3}& := - (\div_{\Gamma_h} (\Hg w)^e, \curl_{\Gamma_h} (\epsilon_c-\epsilon_h))_{\Gamma_h} + (\nabla_{\Gamma_h} w^e,\nabla_{\Gamma_h} (\epsilon_c-\epsilon_h))_{\Gamma_h},\\
I_{4}& := - (\div_{\Gamma_h} (\Hg w)^e, \curl_{\Gamma_h} \epsilon_h)_{\Gamma_h} - ((\Hg w)^e,\HG \epsilon_h)_{\Gamma_h}, \\
I_{5}& :=  ( (\Hg w)^e-\HG w_I,\HG \epsilon_h )_{\Gamma_h} +(\nabla_{\Gamma_h} (w^e-w_I) ,\nabla_{\Gamma_h} \epsilon_h)_{\Gamma_h}.
\end{aligned}
\end{equation*}
The five terms $I_i$ are estimated as follows, each bounded by $h \|w\|_{H^3(\gamma)}\triplenorm{\epsilon_h}$:
\begin{itemize}
\item $I_1$: by the curl--divergence approximation \eqref{eq:geo-1};
\item $I_2$: by the superconvergence estimate \eqref{eq:first-order-tangential};
\item $I_3$: by the properties of the $H^1$-conforming relative \eqref{eq:sNZT-relative};
\item $I_4$: by the consistency estimate \eqref{eq:jump-estimate1};
\item $I_5$: by the interpolation estimates \eqref{eq:Hessian-interpolation} and \eqref{eq:I-L0}.
\end{itemize}

For the remaining term $a_h(w_I, \epsilon_h)$, the finite element scheme \eqref{eq:FEM} gives the decomposition
\begin{equation*}
\begin{aligned}
a_h(w_I, \epsilon_h)=a_h(w_I,\phi_I)-a_h(w_I,\phi_h)
= a_h(w_I,\phi_I) - a(w,\phi) +\sum_{i=1}^3 J_{i},
\end{aligned}
\end{equation*}
where 
\begin{equation*}
J_{1}:= l(w)- l_h(w_I),\quad J_{2}:=s_h^0(w_I,\phi_I),\quad J_{3}:=-s_h^0(w_I,\epsilon_h).
\end{equation*}
The source estimate \eqref{eq:source-estimate2} and the interpolation jump estimates \eqref{eq:value-jump} give
\begin{equation*}
|J_{1}+J_{2}+J_{3}| \lesssim h \| w \|_{H^3(\gamma)} \triplenorm{\epsilon_h} + h^2 \| w \|_{H^3(\gamma)} (\|\bm{f}\|_{L^2(\gamma)} + \|\phi\|_{H^3(\gamma)}).
\end{equation*}
Finally, by \eqref{eq:bilinear-approx-equiv} in Lemma \ref{lm:equivalence-consistency}, it holds that
\begin{equation}\label{eq:ahi-a}
|a_h(w_I,\phi_I) - a(w,\phi)| \lesssim h^2\|w\|_{H^3(\gamma)}\|\phi \|_{H^3(\gamma)}.
\end{equation}

Combining the above estimates with the regularity bound \eqref{eq:regularity-bound}, we arrive at
\begin{equation}
|\epsilon_c^\ell |_{H^1(\gamma)} \lesssim  h\triplenorm{\epsilon_h} + h^2\|\phi\|_{H^3(\gamma)} + h^2\|\bm{f}\|_{L^2(\gamma)}.
\end{equation}
By the relative estimate \eqref{eq:sNZT-relative}, the discrete
Poincar\'e inequality \eqref{eq:poincare},
Theorem~\ref{tm:energy-estimate}, and the regularity $\|\phi\|_{H^3(\gamma)} \lesssim \|\bm{f}\|_{L^2(\gamma)}$, we then obtain
\[
\begin{aligned}
|\epsilon_h|_{H^1_h(\Gamma_h)}
&\lesssim
|\epsilon_h-\epsilon_c |_{H^1_h(\Gamma_h)}
+ |\epsilon_c^\ell |_{H^1(\gamma)}
\lesssim
h\triplenorm{\epsilon_h}
+
h^2\|\phi\|_{H^3(\gamma)}
+
h^2\|\bm f\|_{L^2(\gamma)}
\\
&\lesssim
h^2(
\|\phi\|_{H^3(\gamma)}
+
\|\bm f\|_{L^2(\gamma)}) \lesssim h^2 \|\bm f\|_{L^2(\gamma)}.
\end{aligned}
\]
The estimate \eqref{eq:H1-estimate} then follows from the interpolation estimate \eqref{eq:I-L0} and the triangle inequality.
\end{proof}

\begin{remark}[optimal $L^2$ error estimate for the velocity]
The surface velocity is recovered from the stream function through
$\bm u=\curl_\gamma\phi$, and its discrete counterpart is
$\bm u_h:=\curl_{\Gamma_h}\phi_h$. Comparing the two on $\Gamma_h$ via the
Piola transform $\widebreve{\bm u}$ of $\bm u$, the triangle inequality gives
\[
\|\widebreve{\bm u}-\bm u_h\|_{L^2(\Gamma_h)}
\le
\|\widebreve{\curl_\gamma\phi}-\curl_{\Gamma_h}\phi^e\|_{L^2(\Gamma_h)}
+
\|\curl_{\Gamma_h}(\phi^e-\phi_h)\|_{L^2(\Gamma_h)} .
\]
The first term is $\mathcal O(h^2)\|\phi\|_{H^1(\gamma)}$ by
Corollary~\ref{cor:Piola-curl} (pointwise Piola estimate for the curl),
while the second equals $|\phi^e-\phi_h|_{H^1_h(\Gamma_h)}$ through the pointwise
identity $|\curl_{\Gamma_h}\psi|=|\nabla_{\Gamma_h}\psi|$.
Hence Theorem~\ref{tm:H1-estimate} yields the optimal second-order bound
\[
\|\widebreve{\bm u}-\bm u_h\|_{L^2(\Gamma_h)}
\lesssim h^2 \|\bm{f}\|_{L^2(\gamma)} .
\]
\end{remark}

\section{Numerical experiments} \label{sec:numerical}

In this section, we present numerical experiments to validate our theoretical findings. The implementation is based on the \texttt{iFEM} package \cite{chen2008ifem}.
The surface mesh is uniformly refined by subdividing every triangle into four sub-triangles, with the new nodes projected onto the exact surface $\gamma$.
For an error $E_h$, the reported convergence order is computed as $\log_2(E_h/E_{h/2})$ from consecutive refinement levels.
For all three test cases, we prescribe the manufactured velocity and pressure by
\begin{equation*}
    \bm{u}=\curl_\gamma\phi,
    \qquad
    p(x,y,z)=xyz,
\end{equation*}
from which the force $\bm{f}$ is computed from the strong form \eqref{eq:strong_stokes}. In the FE scheme \eqref{eq:FEM} we then take $\bm{f}_h=\bm{f}^e$.

We consider three benchmark problems. The first example is posed on the unit sphere, which is simply connected, with the exact solution $\phi=r^{-3}(3x^2y-y^3)$. The second example is defined on a torus with major radius $R=1$ and minor radius $r=0.6$, where $\phi=\sin(3\varphi)\cos(3\theta+\varphi)$ in toroidal coordinates $(\theta,\varphi)$; here the manufactured solution $\bm u=\curl_\gamma\phi$ is coexact, so the stream-function formulation applies despite the nontrivial topology of the torus. The third example is posed on the general surface implicitly defined by $\rho(x,y,z)=(x-z^2)^2+y^2+z^2-1=0$, with the exact solution $\phi=y$.

\begin{table}[htbp!]
    \centering
    \caption{Numerical errors and convergence rates for the surface Stokes problem on the unit sphere.}
    \label{tab:stokes_sphere}
    \begin{tabular}{@{}lcccccc@{}}
        \toprule
        Dof & $\|\phi^e-\phi_h\|_{L^2(\Gamma_h)}$ & order
            & $|\phi^e-\phi_h|_{H^1_h(\Gamma_h)}$ & order
            & $|\phi^e-\phi_h|_{H^2_h(\Gamma_h)}$ & order \\
        \midrule
          2562 & 6.66e-02 &      & 4.28e-01 &      & 1.23e+01 &      \\
         10242 & 1.61e-02 & 2.05 & 1.10e-01 & 1.96 & 6.29e+00 & 0.96 \\
         40962 & 4.00e-03 & 2.01 & 2.77e-02 & 1.99 & 3.17e+00 & 0.99 \\
        163842 & 9.97e-04 & 2.00 & 6.95e-03 & 2.00 & 1.59e+00 & 1.00 \\
        \bottomrule
    \end{tabular}
\end{table}

\begin{table}[htbp!]
    \centering
    \caption{Numerical errors and convergence rates for a coexact manufactured surface Stokes solution on the torus.}
    \label{tab:stokes_torus}
    \begin{tabular}{@{}lcccccc@{}}
        \toprule
        Dof & $\|\phi^e-\phi_h\|_{L^2(\Gamma_h)}$ & order
            & $|\phi^e-\phi_h|_{H^1_h(\Gamma_h)}$ & order
            & $|\phi^e-\phi_h|_{H^2_h(\Gamma_h)}$ & order \\
        \midrule
          4096 & 6.59e-01 &      & 3.92e+00 &      & 7.85e+01 &      \\
         16384 & 1.92e-01 & 1.78 & 1.58e+00 & 1.31 & 5.40e+01 & 0.54 \\
         65536 & 5.08e-02 & 1.92 & 4.86e-01 & 1.70 & 3.17e+01 & 0.77 \\
        262144 & 1.29e-02 & 1.98 & 1.30e-01 & 1.90 & 1.67e+01 & 0.92 \\
        \bottomrule
    \end{tabular}
\end{table}

\begin{table}[htbp!]
    \centering
    \caption{Numerical errors and convergence rates for the surface Stokes problem on the general surface.}
    \label{tab:stokes_general}
    \begin{tabular}{@{}lcccccc@{}}
        \toprule
        Dof & $\|\phi^e-\phi_h\|_{L^2(\Gamma_h)}$ & order
            & $|\phi^e-\phi_h|_{H^1_h(\Gamma_h)}$ & order
            & $|\phi^e-\phi_h|_{H^2_h(\Gamma_h)}$ & order \\
        \midrule
          4610 & 2.00e-02 &      & 9.65e-02 &      & 2.89e+00 &      \\
         18434 & 5.36e-03 & 1.90 & 2.75e-02 & 1.81 & 1.63e+00 & 0.83 \\
         73730 & 1.39e-03 & 1.95 & 7.36e-03 & 1.90 & 8.75e-01 & 0.90 \\
        294914 & 3.55e-04 & 1.97 & 1.90e-03 & 1.96 & 4.52e-01 & 0.95 \\
        \bottomrule
    \end{tabular}
\end{table}

\begin{figure}[htbp]
    \centering
    \subcaptionbox{Torus.\label{fig:flow_torus}}{%
        \includegraphics[width=0.50\textwidth]{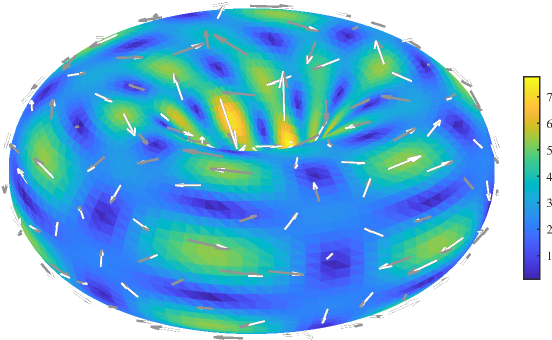}}%
    \hspace{0.06\textwidth}
    \subcaptionbox{General surface.\label{fig:flow_general}}{%
        \includegraphics[width=0.40\textwidth]{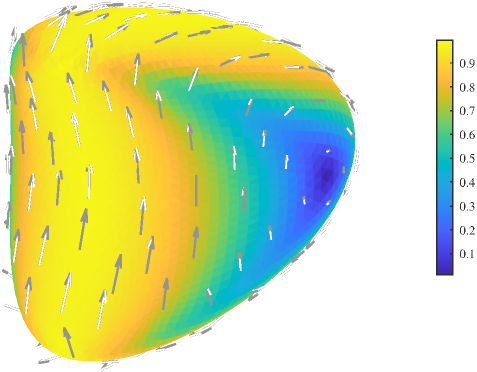}}
    \caption{Computed discrete velocity fields $\bm{u}_h=\curl_{\Gamma_h}\phi_h$. The surface color represents $|\bm{u}_h|$, and the arrows indicate the velocity direction.}
    \label{fig:flow_fields}
\end{figure}

The convergence history for the three benchmarks is reported in Tables~\ref{tab:stokes_sphere}, \ref{tab:stokes_torus}, and~\ref{tab:stokes_general}, and the recovered velocity fields for the torus and the general surface are shown in Figure~\ref{fig:flow_fields}. All three examples exhibit the asymptotic orders $2$, $2$, and $1$ in the $L^2$, broken $H^1$, and broken $H^2$ errors, respectively. For the sphere and the general surface, these results agree with the first-order broken $H^2$ estimate in Theorem~\ref{tm:energy-estimate} (energy norm error estimate) and the second-order broken $H^1$ estimate in Theorem~\ref{tm:H1-estimate} (broken $H^1$ error estimate). The torus results confirm that the same rates persist for the coexact manufactured solution.

To examine the role of the stabilization, let $\phi_h^0$ denote the discrete solution of the sphere problem obtained with the same load functional $l_h$ but after removing $s_h^0$. We also define
\begin{equation*}
\begin{aligned}
\lambda_{\min,h}^{s}
&:=\min_{0\ne v_h\in\mathring{V}_h}
\frac{a_h(v_h,v_h)+s_h^0(v_h,v_h)}
{\|\nabla_{\Gamma_h}^2v_h\|_{L^2(\Gamma_h)}^2
 +\|\nabla_{\Gamma_h}v_h\|_{L^2(\Gamma_h)}^2},\\
\lambda_{\min,h}^{0}
&:=\min_{0\ne v_h\in\mathring{V}_h}
\frac{a_h(v_h,v_h)}
{\|\nabla_{\Gamma_h}^2v_h\|_{L^2(\Gamma_h)}^2
 +\|\nabla_{\Gamma_h}v_h\|_{L^2(\Gamma_h)}^2}.
\end{aligned}
\end{equation*}

Table~\ref{tab:stokes_sphere_stabilization} shows that $\lambda_{\min,h}^{0}$ decreases by approximately a factor of four under each refinement, whereas $\lambda_{\min,h}^{s}$ remains of the same magnitude. This behavior is consistent with the discrete Korn inequality in Lemma~\ref{lem:discrete_korn}: after $s_h^0$ is removed, the coercivity of $a_h$ relative to the broken $H^2$ norm deteriorates as the mesh is refined. Although the unstabilized system remains solvable on each fixed mesh, its broken $H^1$ errors stagnate around $1.5\mathrm{e}{+01}$, while its broken $H^2$ errors approximately double under each refinement. Thus, the unstabilized discrete solutions do not converge to the exact solution. This is in sharp contrast to the optimal rates obtained with the stabilization in Table~\ref{tab:stokes_sphere}.

\begin{table}[htbp]
    \centering
    \caption{Smallest generalized eigenvalues and errors of the unstabilized method for the manufactured surface Stokes problem on the unit sphere.}
    \label{tab:stokes_sphere_stabilization}
    \begin{tabular}{@{}lcccc@{}}
        \toprule
        Dof & $\lambda_{\min,h}^{s}$ & $\lambda_{\min,h}^{0}$
            & $|\phi^e-\phi_h^0|_{H_h^1(\Gamma_h)}$
            & $|\phi^e-\phi_h^0|_{H_h^2(\Gamma_h)}$ \\
        \midrule
           642 & 3.9161e-02 & 5.5632e-03 & 1.5307e+01 & 2.0907e+02 \\
          2562 & 3.3681e-02 & 1.4432e-03 & 1.5567e+01 & 4.1812e+02 \\
         10242 & 3.2203e-02 & 3.6526e-04 & 1.5630e+01 & 8.3606e+02 \\
        \bottomrule
    \end{tabular}
\end{table}

\section{Conclusion}\label{sec:conclusion}
We have developed and analyzed a primal nonconforming discretization of the surface stream-function Stokes problem based on the surface Morley element. The scheme is intrinsic to the discrete surface, uses no discrete Gaussian-curvature approximation, and involves no tunable penalty parameter. Under the minimal regularity assumptions it attains optimal-order convergence in the broken $H^2$ and $H^1$ norms and a second-order $L^2$ estimate for the recovered velocity.
The core of the analysis, and the main contribution of this work, is the geometric-consistency machinery of Section~\ref{sec:geo-error}. Its foundation is a single normal-separated estimate (Theorem~\ref{thm:normal-separated}): expressing a surface integrand as a function of the surface normal and Taylor expanding in that variable exposes an integral cancellation in the first-order term, so that the pointwise first-order error of the discrete normal does not degrade the consistency order. 

This one estimate resolves a broad class of geometric errors within a common framework. It recovers the classical $\bP_h\bnu$-type estimate as a special case (Corollary~\ref{cor:Phn}), yields superconvergence for first-order tangential forms (Corollary~\ref{cor:first-order-tangential}), and delivers the second-order interior and boundary consistency of the tensorial forms produced by the stream-function formulation---the trace-free Hessian, the Stokes-type tensor Green identity, and the conormal fluxes (Lemmas~\ref{lem:neilan-volume-estimate}, \ref{lem:stokes-curl-div}, and \ref{lem:weak-conormal-flux}, together with their corollaries). The error analysis of the Morley scheme is then, in large part, an application of these estimates.

The normal-separated argument is not tied to the stream-function Stokes problem or to the Morley element: it applies to any surface finite element form whose geometric error enters through the discrete normal. Extending the mechanism to higher-order geometric quantities, in particular to the discrete curvature, is the subject of work in preparation.

\section*{Statements and Declarations}

\paragraph{Funding.}
This work was supported by the National Natural Science Foundation of China under Grant Nos.~12571383 and 12288101.

\paragraph{Competing interests.}
The authors have no relevant financial or non-financial interests to disclose.

\appendix
\section{Proof of Lemma \ref{lm:Pnjump}} \label{sec:appendix_proof}
\begin{proof}
Let $\bm{t}_e$ be the unit tangent vector along $e = \partial K_1 \cap \partial K_2$, oriented such that the discrete outward co-normals are $\bm{n}_1 = \bm{t}_e \times \bm{\nu}_1$ and $\bm{n}_2 = -\bm{t}_e \times \bm{\nu}_2$. The jump is $\llbracket \bn_h \rrbracket = \bm{t}_e \times \bm{d}$, where $\bm{d} := \bm{\nu}_1 - \bm{\nu}_2$. Decomposing $\bm{d}$ into its exact normal and tangential components, $\bm{d} = (\bm{d} \cdot \bm{\nu}) \bm{\nu} + \mathbf{P}\bm{d}$, yields
\begin{equation}\label{eq:P_jump_decomp}
\mathbf{P}\llbracket \bn_h \rrbracket = (\bm{d} \cdot \bm{\nu}) \mathbf{P}(\bm{t}_e \times \bm{\nu}) + \mathbf{P}(\bm{t}_e \times \mathbf{P}\bm{d}).
\end{equation}

For the first term, using the identity $2\bm{\nu}_i \cdot \bm{\nu} = 2 - |\bm{\nu}_i - \bm{\nu}|^2$ and the approximation property $|\bm{\nu}_i - \bm{\nu}| \lesssim h$, we obtain $|\bm{d} \cdot \bm{\nu}| = \frac{1}{2}\big| |\bm{\nu}_2 - \bm{\nu}|^2 - |\bm{\nu}_1 - \bm{\nu}|^2 \big| \lesssim h^2$. Since $|\mathbf{P}(\bm{t}_e \times \bm{\nu})| \le 1$, this term is bounded by $\mathcal{O}(h^2)$. 

For the second term, we decompose $\bm{t}_e = \mathbf{P}\bm{t}_e + (\bm{t}_e \cdot \bm{\nu})\bm{\nu}$. Because $\mathbf{P}\bm{t}_e$ and $\mathbf{P}\bm{d}$ are both purely tangential to $\gamma$, their cross product is parallel to $\bm{\nu}$, which implies $\mathbf{P}(\mathbf{P}\bm{t}_e \times \mathbf{P}\bm{d}) = \bm{0}$. Consequently,
$$
|\mathbf{P}(\bm{t}_e \times \mathbf{P}\bm{d})| = \big| \mathbf{P} \big( (\bm{t}_e \cdot \bm{\nu})(\bm{\nu} \times \mathbf{P}\bm{d}) \big) \big| \le |\bm{t}_e \cdot \bm{\nu}| |\mathbf{P}\bm{d}| \lesssim h^2,
$$
where we utilized $|\mathbf{P}\bm{d}| \le |\bm{d}| \lesssim h$ and $|\bm{t}_e \cdot \bm{\nu}| \lesssim h$. Substituting both bounds back into \eqref{eq:P_jump_decomp} completes the proof.
\end{proof}

\section{Proof of Corollary \ref{cor:Piola-curl}} \label{sec:appendix_piola_curl}
\begin{proof}
Write $\bm a:= (\nabla_\gamma w)^e$ and $\bm e:=\bm\nu_h-\bm\nu$. Since $\bm\nu$ and $\bm\nu_h$ are unit vectors, we have $|\bm{e}| \lesssim h$, $ |1-\bm\nu\cdot\bm\nu_h| \lesssim h^2$, and $|\bm\nu \cdot\bm e| = \tfrac12|\bm e|^2\lesssim h^2$. Moreover, $|\bnu_h \times \bnu| = |\bnu_h \times \bm e| \lesssim h$.

Since $\curl_{\Gamma_h}w^e=\bm\nu_h\times\nabla_{\Gamma_h}w^e$, Lemma~\ref{lm:Piola-derivative} gives
\begin{equation}\label{eq:piola-curl-step1}
\curl_{\Gamma_h}w^e
=\bm\nu_h\times\widebreve{\nabla_\gamma w}
+\bm\nu_h\times\big(\nabla_{\Gamma_h}w^e-\widebreve{\nabla_\gamma w}\big)
=\bm\nu_h\times\widebreve{\nabla_\gamma w}+\mathcal O(h^2)|\bm a|.
\end{equation}
It remains to compare $\widebreve{\curl_\gamma w}$ with $\bm\nu_h\times\widebreve{\nabla_\gamma w}$.
From the Piola formula~\eqref{def:Piola_invp} together with $|d|,\,|1-\mu_h|,\,|1-\bm\nu^e\cdot\bm\nu_h|\lesssim h^2$, every tangential field $\bm b$ on $\gamma$ satisfies
\[
\breve{\bm{b}} = \mathscr P_{\bm p^{-1}}\bm b=\bm b^e-\bm\nu(\bm e\cdot\bm b^e)+\mathcal O(h^2)|\bm b|.
\]
Applying this respectively to the tangential fields $\bm a$ and $\bm\nu \times\bm a$, and using $(\curl_\gamma w)^e=\bm\nu \times\bm a$, we have
\[
\begin{aligned}
\widebreve{\nabla_\gamma w} &= \bm{a} - \bnu (\bm{e} \cdot \bm{a}) + \mathcal O(h^2)|\bm a|, \\
\widebreve{\curl_\gamma w} &= \bm\nu \times\bm a - \bm\nu \big(\bm e\cdot(\bm\nu \times\bm a)\big)+\mathcal O(h^2)|\bm a|.
\end{aligned}
\]
Consequently, 
$$
\begin{aligned}
\bm\nu_h\times\widebreve{\nabla_\gamma w} &= \bm\nu_h\times\bm a - \bnu_h \times \bnu (\bm{e} \cdot \bm{a}) + \mathcal O(h^2)|\bm a| \\
&= \bnu_h\times\bm a + \mathcal O(h^2)|\bm a|
=\bm\nu \times\bm a + \bm e\times\bm a+\mathcal O(h^2)|\bm a|.
\end{aligned}
$$
Then, by subtracting, we obtain
\[
\widebreve{\curl_\gamma w} - \bm\nu_h\times\widebreve{\nabla_\gamma w}
=-\bm e\times\bm a-\bm\nu \big(\bm e\cdot(\bm\nu \times\bm a)\big)+\mathcal O(h^2)|\bm a|
=-\bP(\bm e\times\bm a)+\mathcal O(h^2)|\bm a|.
\]
Finally, $|\bnu \cdot \bm{e}| \lesssim h^2$ implies that $\bm e$ is tangential to $\gamma$ up to $\mathcal O(h^2)$, so the cross product $\bm e\times\bm a$ is normal up to $\mathcal O(h^2)$, whence $|\bP(\bm e\times\bm a)|\lesssim h^2|\bm a|$. Combining this with~\eqref{eq:piola-curl-step1} proves~\eqref{eq:Piola-curl}.
\end{proof}

    


\bibliographystyle{unsrt}
\bibliography{surface}

\begin{thebibliography}{10}

\bibitem{jankuhn2018incompressible}
Thomas Jankuhn, Maxim~A. Olshanskii, and Arnold Reusken.
\newblock Incompressible fluid problems on embedded surfaces: Modeling and
  variational formulations.
\newblock {\em Interfaces and Free Boundaries}, 20(3):353--377, 2018.

\bibitem{dziuk2013finite}
Gerhard Dziuk and Charles~M Elliott.
\newblock Finite element methods for surface {PDEs}.
\newblock {\em Acta Numerica}, 22:289--396, 2013.

\bibitem{olshanskii2018finite}
Maxim~A. Olshanskii, Annalisa Quaini, Arnold Reusken, and Vladimir Yushutin.
\newblock A finite element method for the surface {Stokes} problem.
\newblock {\em SIAM Journal on Scientific Computing}, 40(4):A2492--A2518, 2018.

\bibitem{jankuhn2021error}
Thomas Jankuhn, Maxim~A. Olshanskii, Arnold Reusken, and Alexander Zhiliakov.
\newblock Error analysis of higher order trace finite element methods for the
  surface {Stokes} equation.
\newblock {\em Journal of Numerical Mathematics}, 29(3):245--267, 2021.

\bibitem{brandner2022finite}
Philip Brandner, Thomas Jankuhn, Simon Praetorius, Arnold Reusken, and Axel
  Voigt.
\newblock Finite element discretization methods for velocity--pressure and
  stream function formulations of surface {Stokes} equations.
\newblock {\em SIAM Journal on Scientific Computing}, 44(4):A1807--A1832, 2022.

\bibitem{hardering2024parametric}
Hanne Hardering and Simon Praetorius.
\newblock Parametric finite-element discretization of the surface {Stokes}
  equations: inf-sup stability and discretization error analysis.
\newblock {\em IMA Journal of Numerical Analysis}, 45(5):2948--2987, 2025.

\bibitem{bonito2020divergence}
Andrea Bonito, Alan Demlow, and Martin Licht.
\newblock A divergence-conforming finite element method for the surface
  {S}tokes equation.
\newblock {\em SIAM Journal on Numerical Analysis}, 58(5):2764--2798, 2020.

\bibitem{lederer2020divergence}
Philip~L Lederer, Christoph Lehrenfeld, and Joachim Sch{\"o}berl.
\newblock Divergence-free tangential finite element methods for incompressible
  flows on surfaces.
\newblock {\em International Journal for Numerical Methods in Engineering},
  121(11):2503--2533, 2020.

\bibitem{demlow2024tangential}
Alan Demlow and Michael Neilan.
\newblock A tangential and penalty-free finite element method for the surface
  {S}tokes problem.
\newblock {\em SIAM Journal on Numerical Analysis}, 62(1):248--272, 2024.

\bibitem{demlow2025taylor}
Alan Demlow and Michael Neilan.
\newblock A {Taylor}--{Hood} finite element method for the surface {Stokes}
  problem without penalization.
\newblock {\em SIAM Journal on Numerical Analysis}, 64(2):565--600, 2026.

\bibitem{nochetto2025surface}
Ricardo~H. Nochetto and Mansur Shakipov.
\newblock Surface {Stokes} without inf-sup condition.
\newblock arXiv preprint arXiv:2508.13342, 2025.

\bibitem{reusken2020stream}
Arnold Reusken.
\newblock Stream function formulation of surface {S}tokes equations.
\newblock {\em IMA Journal of Numerical Analysis}, 40(1):109--139, 2020.

\bibitem{brueers2026releasing}
Tim Br{\"u}ers, Christoph Lehrenfeld, Tim van Beeck, and Max Wardetzky.
\newblock Releasing the pressure: High-order surface flow discretizations via
  discrete {Helmholtz}--{Hodge} decompositions.
\newblock arXiv preprint arXiv:2603.27714, 2026.

\bibitem{brandner2020finite}
Philip Brandner and Arnold Reusken.
\newblock Finite element error analysis of surface {Stokes} equations in stream
  function formulation.
\newblock {\em ESAIM: Mathematical Modelling and Numerical Analysis},
  54(6):2069--2097, 2020.

\bibitem{neilan2024c}
Michael Neilan and Hongzhi Wan.
\newblock A {$C^0$} interior penalty method for the stream function formulation
  of the surface {Stokes} problem.
\newblock {\em ESAIM: Mathematical Modelling and Numerical Analysis},
  59(2):1177--1211, 2025.

\bibitem{larsson2017continuous}
Karl Larsson and Mats Larson.
\newblock A continuous/discontinuous {G}alerkin method and a priori error
  estimates for the biharmonic problem on surfaces.
\newblock {\em Mathematics of Computation}, 86(308):2613--2649, 2017.

\bibitem{cai2024continuous}
Ying Cai, Hailong Guo, and Zhimin Zhang.
\newblock Continuous linear finite element method for biharmonic problems on
  surfaces.
\newblock {\em Journal of Scientific Computing}, 107:39, 2026.

\bibitem{elliott2019second}
Charles Elliott, Hans Fritz, and Graham Hobbs.
\newblock Second order splitting for a class of fourth order equations.
\newblock {\em Mathematics of Computation}, 88(320):2605--2634, 2019.

\bibitem{stein2019mixed}
Oded Stein, Eitan Grinspun, Alec Jacobson, and Max Wardetzky.
\newblock A mixed finite element method with piecewise linear elements for the
  biharmonic equation on surfaces.
\newblock {\em arXiv preprint arXiv:1911.08029}, 2019.

\bibitem{walker2022kirchhoff}
Shawn~W Walker.
\newblock The {K}irchhoff plate equation on surfaces: the surface
  {H}ellan--{H}errmann--{J}ohnson method.
\newblock {\em IMA Journal of Numerical Analysis}, 42(4):3094--3134, 2022.

\bibitem{morley1968triangular}
L.~S.~D. Morley.
\newblock The triangular equilibrium element in the solution of plate bending
  problems.
\newblock {\em The Aeronautical Quarterly}, 19(2):149--169, 1968.

\bibitem{wang2006morley}
Ming Wang and Jinchao Xu.
\newblock The {Morley} element for fourth order elliptic equations in any
  dimensions.
\newblock {\em Numerische Mathematik}, 103(1):155--169, 2006.

\bibitem{wu2025stabilized}
Shuonan Wu and Hao Zhou.
\newblock A stabilized nonconforming finite element method for the surface
  biharmonic problem.
\newblock {\em SIAM Journal on Numerical Analysis}, 63(4):1642--1665, 2025.

\bibitem{Yakov2025blowup}
Yakov {Berchenko-Kogan} and Evan~S. Gawlik.
\newblock Blow-up {{Whitney}} forms, shadow forms, and {{Poisson}} processes.
\newblock {\em Results in Applied Mathematics}, 25:100529, 2025.

\bibitem{demlow2009higher}
Alan Demlow.
\newblock Higher-order finite element methods and pointwise error estimates for
  elliptic problems on surfaces.
\newblock {\em SIAM Journal on Numerical Analysis}, 47(2):805--827, 2009.

\bibitem{bonito2020finite}
Andrea Bonito, Alan Demlow, and Ricardo~H Nochetto.
\newblock Finite element methods for the {L}aplace--{B}eltrami operator.
\newblock In {\em Handbook of Numerical Analysis}, volume~21, pages 1--103.
  Elsevier, 2020.

\bibitem{dziuk1988finite}
Gerhard Dziuk.
\newblock {\em Finite elements for the {Beltrami} operator on arbitrary
  surfaces}.
\newblock Springer, 1988.

\bibitem{hansbo2020analysis}
Peter Hansbo, Mats~G Larson, and Karl Larsson.
\newblock Analysis of finite element methods for vector laplacians on surfaces.
\newblock {\em IMA Journal of Numerical Analysis}, 40(3):1652--1701, 2020.

\bibitem{chen2008ifem}
Long Chen.
\newblock i{FEM}: an innovative finite element methods package in {MATLAB}.
\newblock {\em Preprint, University of Maryland}, 2008.

\end{thebibliography}

\end{document}